\documentclass[11pt,a4paper]{amsart}
\usepackage{amsmath}
\usepackage{amsfonts}
\usepackage{amssymb}
\usepackage{amsthm}

\theoremstyle{plain}
\newtheorem{thm}{Theorem}[section]
\newtheorem{prop}{Proposition}[section]
\newtheorem{lem}[thm]{Lemma}

\theoremstyle{definition}
\newtheorem{defn}{Definition}[section]

\newtheorem{exmp}{Example}[section]

\theoremstyle{remark}
\newtheorem{rem}{Remark}[section]

\def\N{\mathbb{N}}
\def\C{\mathbb{C}}
\def\R{\mathbb{R}}
\def\Z{\mathbb{Z}}
\def\T{\mathfrak{X}}
\def\Op{\text{\upshape{Op}}}
\def\Opd{\text{\upshape{OP}}}

\def\S{S}
\def\Sc{S_{sc}}

\def\eps{\epsilon}

\numberwithin{equation}{section}

\begin{document}
\title{Pseudo-differential calculus on homogeneous trees}
\author{Etienne Le~Masson}
\address{Universit\'e Paris-Sud 11, Math\'ematiques, B\^at. 425, 91405 ORSAY CEDEX, FRANCE}
\email{etienne.lemasson@math.u-psud.fr}
\date{}

\begin{abstract}
In the objective of studying concentration and oscillation properties of eigenfunctions of the discrete Laplacian on regular graphs, we construct a pseudo-differential calculus on homogeneous trees, their universal covers. We define symbol classes and associated operators. We prove that these operators are bounded on $L^2$ and give adjoint and product formulas. Finally we compute the symbol of the commutator of a pseudo-differential operator with the Laplacian. 
\end{abstract}

\maketitle

\section{Introduction}
The study of the Laplacian on graphs has attracted much attention in the field of quantum chaos since the work of Kottos and Smilansky (\cite{KS97,KS99}). Most of the interest has been on metric graphs, specifically at the spectrum of the Laplacian in this context, and to a lesser extent at eigenfunctions (\cite{KMW03,BKW04,BKS07}). Recently (\cite{Smi10,Smi07}), the discrete Laplacian on regular graphs has been suggested as a good simplified model for quantum chaos.

The aim of this paper is to provide a tool to study concentration and oscillation properties of eigenfunctions of the discrete Laplacian on regular graphs. It is known that on compact hyperbolic surfaces (and more generally on compact Riemannian manifolds with strictly negative curvature) the eigenfunctions of the Laplacian are associated with probability densities, most of which tend to a uniform distribution in the high-frequency limit. This result called quantum ergodicity is proved with the help of pseudo-differential calculus on manifolds, and the proof uses the ergodicity of the geodesic flow (\cite{Sni74,Zel87,CdV85}). A recent result (\cite{BL}) suggests a similar behaviour should exist for eigenfunctions on finite regular graphs when the high-frequency limit is replaced by a large spatial scale limit (that is when the size of the graphs tends to infinity). 
But one of the difficulties preventing from going further is that, even though ergodic dynamics can be defined on these graphs, to our knowledge no pseudo-differential calculus is available in this context.
We therefore construct in this article a pseudo-differential calculus on the universal covers of regular graphs (that is regular trees). It is analogous to the one constructed by Zelditch in \cite{Zel86} for hyperbolic surfaces, in the sense that it is based on the Fourier-Helgason transform and thus is adapted to the geometry.

The outline of the paper is as follows. In section \ref{s:harmonic} we recall some necessary elements of harmonic analysis on homogeneous trees. In particular we define the Fourier-Helgason transform and give some of its properties. 

In section \ref{s:operators} we define operators $\Op(a)$ on the $q+1$-regular tree $\T$, associated with functions $a(x,\omega,s)$, where $x\in \T$, $\omega \in \Omega$ is an element of the boundary of $\T$, and $s \in [0,\pi/\log q] \subset \R$ is a spectral parameter (see section \ref{s:harmonic} for definitions). We introduce a general symbol class $\S$ and a subset $\Sc$ of \emph{semi-classical} symbols depending on a small parameter $\epsilon > 0$ controlling the variation of the symbol on the tree. We also prove in this section that $S$ and $\Sc$ are algebras, and closed under the action of two operators related to the dynamics of the tree (the shift and the transfer operator associated, see definition \ref{shift operator}). Using these properties, we then show that $\Sc$ contains nontrivial examples of symbols.

In section \ref{s:decay} we adapt an existing result of \cite{CS99} on the Fourier-Helgason transform on trees to show that the kernels of the operators we defined have a rapid decay property away from the diagonal. 

This fact will be essential to prove that our pseudo-differential operators are bounded as operators from $L^2(\T)$ to $L^2(\T)$, which is done in section \ref{s:continuity}. The theorem can be stated in the following way

\begin{thm}
 Let $a = a(x,\omega,s) \in \S$ be a symbol, then the pseudo-differential operator $\Op(a)$ can be extended as a bounded operator from $L^2(\T)$ to $L^2(\T)$ and we have the inequality
\[ \|\Op(a) \|_{L^2 \rightarrow L^2} \leq C \left( \|a\|_\Omega + \sum_{k=0}^4 \|\partial_s^k a \|_\infty \right), \]
where $\|a\|_\Omega$ is a norm associated with the regularity of $a$ in $\omega$.
\end{thm}

In section \ref{s:adjoint} we consider semi-classical symbols and prove an adjoint formula
\begin{thm}
Let $a = a_\eps \in \Sc$. Let $\Op(a)^*$ be the adjoint of the pseudo-differential operator $\Op(a)$ associated with the symbol $a$. Then
\[ \| \Op(a)^* - \Op(\overline{a}) \|_{L^2 \rightarrow L^2} = o(1), \]
where $o(1) \rightarrow 0$ when $\epsilon \rightarrow 0$.
\end{thm}
and a product formula
\begin{thm}
Let $a \in \S$ and $b = b_\eps \in \Sc$. Then
\[ \| \Op(a)\Op(b) - \Op(ab) \|_{L^2 \rightarrow L^2} = o(1), \]
where $o(1) \rightarrow 0$ when $\epsilon \rightarrow 0$.
\end{thm}
We also give (theorem \ref{t:product optimal}) a proof leading to a stronger conclusion (replacing $o(1)$ by $O(\eps)$) when the symbols satisfy a more restrictive (but reasonable) condition.
The expansions in all these formulas are done only to the first order, as it is not clear what higher orders would mean in this setting.

In section \ref{s:commutator} we give a formula for the symbol of the commutator of a pseudo-differential operator $\Op(a)$ and the discrete Laplacian on the tree, which makes evident a relationship with the dynamics.

\subsection*{Acknowledgements}
The author would like to thank his advisor Nalini Anantharaman for her guidance, corrections, and all the discussions that led to the mathematical ideas in this paper. The author would also like to thank Yves Colin de Verdi\`ere for helpful discussions and encouragements.

\section{Harmonic analysis on homogeneous trees}\label{s:harmonic}
The basic elements of harmonic analysis on homogeneous trees that we will need can be found in \cite{CMS98} and in the first two chapters of \cite{FTN91}. We will recall them briefly in this section.

\subsection{The tree and its boundary}
Let $\T$ be the $q+1$-homogeneous tree. We see it as a metric space $(\T,d)$ with $\T$ the set of vertices and $d : \T \times \T \rightarrow \N$ a distance such that $d(x,y)=1$ if and only if the two vertices $x$ and $y$ are neighbours. We denote by $[x,y]$ the unique \emph{geodesic path} joining two arbitrary vertices $x$ and $y$, that is the sequence of vertices $(x_0, x_1,\ldots,x_n)$ such that $x_0=x$, $x_n=y$, and $d(x_i,x_j) = |i - j |$ for all $i,j \in \{0, \ldots, n \}$. We fix an arbitrary reference point $o \in \T$, and $|x|$ will denote the distance from $o$ of the vertex $x$.

The boundary $\Omega$ of the tree can be defined in one of the following equivalent ways. 
\begin{itemize}
 \item It is the set of half geodesics $(o,x_1,x_2,\ldots) \in \T^\N$ starting at $o$.
 \item It is also the set of equivalence classes of half geodesics, where two half geodesics $(x_0,x_1,x_2,\ldots)$ and $(y_0,y_1,y_2,\ldots)$ are equivalent if they meet at infinity, that is if there exists $k,N \in \N$ such that for all $n \geq N, x_n = y_{n+k}$.
\end{itemize}
The half geodesic starting at a point $x$ and equivalent to $\omega$ will be denoted by $[x,\omega)$. 

The topology on the boundary $\Omega$ is described by the basis of open \emph{cylinders} 
\[ \Omega(x,y) = \{ \omega \in \Omega, [x,y] \text{ is a subsequence of } [x,\omega) \}, \] for all $x,y \in \T$.

Let us introduce also the set $\Omega_n(x,\omega)$, for every $x \in \T$, $\omega \in \Omega$ and $n \in \N$. It is defined by
\begin{equation}
 \Omega_n(x,\omega) = \Omega(x,x_n),
\end{equation}
where $d(x,x_n) = n$ and $[x,x_n]$ is a subsequence of $[x,\omega)$.

For every $x \in \T$, a measure $\nu_x$ can be defined on the boundary by noticing that for every $n \in \N$, $\{ \Omega(x,y)\}_{y:d(x,y) = n}$ forms a partition of $\Omega$ into $(q+1)q^{n-1}$ sets. For every $x \in \T$, there is a probability measure $\nu_x$ such that on these partitions we have \[ \nu_x(\Omega(x,y)) =  \frac{1}{(q+1)q^{n-1}}. \]
When $x=o$, we will generally denote this measure by $\nu$ instead of $\nu_o$.

\subsection{Horocycles and height functions}
The analogue of the Busemann functions of Riemannian geometry are the height functions $h_\omega : \T \rightarrow \Z$ defined for all $\omega = (x_0, x_1,\ldots)$ in $\Omega$ with $x_0 = o$ by
\[ h_\omega(x) = \lim_{m \rightarrow + \infty} (m - d(x,x_m)). \]
Note that these functions are normalized so that $h_\omega(o) = 0$. We can then define the $\omega$-horocycles to be the level sets of $h_\omega$ : for all $k \in \Z$ 
\[ \mathfrak{h}(\omega,k) = \{x \in \T : h_\omega(x) = k \}. \]

We will mostly use another description of the height functions. For every $x \in \T$ and $\omega \in \Omega$, we denote by $c(x,\omega)$ the last point lying on $[o,\omega)$ in the geodesic path $[o,x]$ and call it the \emph{confluence point} of $x$ and $\omega$. We have
\begin{equation}\label{e:confluence}
  h_\omega(x) = 2 |c(x,\omega)| - |x|.
\end{equation}
This allows us to introduce a partition of $\Omega$. For $x \in \T$, it is the family of sets defined by
\begin{equation}\label{e:decomposition omega}
 E_i(x) = \{ \omega \in \Omega : |c(x,\omega)| = i \},
\end{equation}
for all $i \in \N$ such that $0 \leq i \leq |x|$. The set $E_{|x|}(x)$ is also simply denoted by $E(x)$. We have
\begin{equation}
 \nu(E_i(x)) \leq q^{-i}
\end{equation}
for every $0 \leq i \leq |x|$. More precisely we have $\nu(E_0(x)) = \frac{q}{q+1}$, $\nu(E(x)) = \frac{q}{q+1} q^{-|x|}$ and $\nu(E_j(x)) = \frac{q-1}{q+1}q^{-j}$ for every $i\in \N$ such that $0 < i < |x|$.

We define a family of averaging operators $\mathcal{E}_n$ acting on bounded measurable functions $\nu$ defined on $\Omega$, by $\mathcal{E}_{-1} = 0$ and, when $n \geq 0$,
\begin{equation}\label{e:mean operator}
 \mathcal{E}_n \eta(\omega) = \frac{1}{\nu(\Omega_n(o,\omega))} \int_{\Omega_n(o,\omega)} \eta(\omega') d\nu(\omega').
\end{equation}

\subsection{Change of reference point}
Note that all the previous objects can be defined with respect to another reference point. We will often use a superscript to denote the change of reference point: for a new reference point $x_0$ we would write $c^{x_0}(x,\omega)$, $E_i^{x_0}(x)$, $\mathcal{E}_n^{x_0}$\ldots{} However, we will write directly $h_\omega(x) - h_\omega(x_0)$ rather than $h^{x_0}_\omega(x)$.

Note the Radon-Nikodym derivative
\begin{equation}
 \frac{d\nu_y}{d\nu_x}(\omega) = q^{h_\omega(y) - h_\omega(x)},
\end{equation}
for all $x,y \in \T$ and $\omega \in \Omega$.
\subsection{The Fourier-Helgason transform}
Let $f : \T \rightarrow \C$ be a finite support function. The \emph{Fourier-Helgason transform} of $f$, denoted by $\hat f$ or $\mathcal{H}f$ is given by the formula
\[ \mathcal{H}f(\omega,s) = \hat f(\omega,s) = \sum_{y \in T} f(y) q^{(1/2 + is)h_\omega(y)} \]
for every $\omega \in \Omega$ and $s \in \mathbb{T}$,  with $\mathbb{T} = \Z / 2\tau \Z$, $\tau = \frac{\pi}{\log q}$. Note that here $\log(t)$ is the natural logarithm of $t \in \R$.
We have an inverse formula given by
\begin{equation}\label{e:inverse}
  f(x) = \int_\Omega \int_{\mathbb{T}} q^{(1/2 - is)h_\omega(x)} \hat{f}(\omega,s)  d\nu(\omega)  d\mu(s).
\end{equation}
The measure $d\mu(s)$ is called the \emph{Plancherel measure}, defined by
\begin{equation}\label{e:plancherel}
 d\mu(s) = c_P|\mathbf{c}(s)|^{-2},
\end{equation}
where $c_P = \frac{q \log(q)}{4 \pi (q+1)}$ and for every $z \in \C \setminus \tau \Z$,
\begin{equation}\label{c(z)}
 \mathbf{c}(z) = \frac{q^{1/2}}{q+1} \frac{q^{1/2 + iz} - q^{-1/2 -iz}}{q^{iz} - q^{-iz}}.
\end{equation}
In particular
\begin{equation}\label{c(s)^-2}
 |c(s)|^{-2} = \frac{(q+1)^2}{q} \frac{4 \sin^2(s \log q)}{q + q^{-1} - 2 \cos(2s\log q)}.
\end{equation}

There is a Plancherel formula: if $f$ and $g$ are two finitely supported functions on $\T$, then
\[ \sum_{x \in \T} f(x)\overline{g(x)} = \int_\mathbb{T}\int_\Omega \hat f(\omega,s) \overline{\hat g(\omega,s)} d\nu(\omega) d\mu(s). \]
The Fourier-Helgason transform $\mathcal{H}$ can therefore be extended to an isometry from $L^2(\T)$ to its image in $L^2(\Omega \times \mathbb{T})$, which is the subspace of functions $F$ satisfying the \emph{symmetry condition}
\begin{equation}\label{e:symmetry}
 \int_\Omega q^{(\frac{1}{2}-is) h_\omega(x)} F(\omega,s) d\nu(\omega) 
  =  \int_\Omega q^{(\frac{1}{2}+is) h_\omega(x)} F(\omega,-s) d\nu(\omega)
\end{equation}

\subsection{Spherical functions}
The spherical functions are defined, for $z \in \C$ by
\[ \phi_z(x) = \int_\Omega q^{(\frac{1}{2}+iz)h_\omega(x)} d\nu(\omega). \]
They are given explicitly by
\begin{equation}
 \phi_z(x) = 
 \begin{cases}
  \left(\frac{q-1}{q+1}|x|+1\right)q^{-\frac{|x|}{2}} & \text{if } z \in 2\tau\Z \\
  \left(\frac{q-1}{q+1}|x|+1\right)q^{-\frac{|x|}{2}}(-1)^{|x|} & \text{if } z \in \tau + 2\tau\Z \\
  c(z)q^{(iz-\frac{1}{2})|x|} + c(-z) q^{-iz-\frac{1}{2})|x|} & \text{if } z \in \C \setminus \tau\Z
 \end{cases}
\end{equation}

\subsection{Rapidly decreasing functions}
A fundamental property used in this article will be the link between rapid decay of functions and regularity of their Fourier-Helgason transform. To make this link we need some definitions.

\begin{defn} 
A function $f : \T \rightarrow \C$ is said to be \emph{rapidly decreasing} if for every $k \in \N$, there exists $C_k > 0$ such that
\[ |f(x)| \leq C_k \frac{q^{-\frac{|x|}{2}}}{(1+|x|)^k}. \]
We denote by $\mathcal{S}(\T)$ the space of rapidly decreasing functions on $\T$.
\end{defn}
\begin{defn}
A function $F : \Omega \times \mathbb{T} \rightarrow \C$ is in $C^\infty(\Omega \times \mathbb{T})$ if
\begin{enumerate}
\item $\forall l \in \N \quad \partial_s^l F(\omega,s) \in C(\Omega \times \mathbb{T})$, where $C(\Omega \times \mathbb{T})$ denotes the space of continuous functions on $\Omega \times \mathbb{T}$.
\item $\forall k,l \in \N, \exists C_{k,l} > 0, \forall n \in \N \quad \| \partial_s^k(F - \mathcal{E}_nF) \|_\infty \leq C_{k,l} (n+1)^{-l}$
\end{enumerate}
We will denote by $C^\infty(\Omega \times \mathbb{T})^\flat$ the space of functions in $C^\infty(\Omega \times \mathbb{T})$ satisfying the symmetry condition $\eqref{e:symmetry}$.
\end{defn}
We then have the following theorem, proved in \cite{CS99}
\begin{thm}\label{t:Paley Wiener 2}
 The Fourier-Helgason transform $\mathcal{H}$ is an isomorphism from $\mathcal{S}(\T)$ onto $C^\infty(\Omega \times \mathbb{T})^\flat$.
\end{thm}

\subsection{About the symmetry condition} The symmetry condition $\eqref{e:symmetry}$ does not behave well when we take the product of two functions. This means that the product of two functions satisfying this condition does not satisfy it anymore in general. We would like to avoid working with this condition.
For this purpose, let us examnie the inversion formula \eqref{e:inverse} and note that we have
\begin{align*}
 f(x) &= \int_\Omega \int_{\mathbb{T}}  q^{(\frac12 - is)h_\omega(x)} \hat{f}(\omega,s) d\nu(\omega) d\mu(s) \\
 &= 2 \int_\Omega \int_0^\tau q^{(\frac12 - is)h_\omega(x)} \hat{f}(\omega,s) d\nu(\omega) d\mu(s),
\end{align*}
using the fact that $d\mu(-s) = d\mu(s)$ and the symmetry condition \eqref{e:symmetry}.
We can therefore work with $s \in [0,\tau]$ instead of $s \in \mathbb{T}$. We will see the Fourier-Helgason transform $\mathcal{H}$ as an isometry from $L^2(\T)$ onto $L^2(\Omega \times [0,\tau])$, and we will not need to refer to the symmetry condition anymore. 

But then, theorem \ref{t:Paley Wiener 2} cannot be used in this form. The only part of this theorem we will be interested in is the fact that if a function $F(\omega,s)$ is in $C^\infty(\Omega \times \mathbb{T})$ and satisfies the symmetry condition, then $\mathcal{H}^{-1}F$ is rapidly decreasing. To maintain this property without the symmetry condition, we will need to ask that $\partial_s^k F(\omega,0) = \partial_s^k F(\omega,\tau) = 0$ for all $k \in \N$. This is what we will do when we define symbol classes in section \ref{ss:operators}. The adaptation of the proof of theorem \ref{t:Paley Wiener 2} to the new context will be done in section \ref{s:decay}.

\section{Operators and symbol classes}\label{s:operators}

\subsection{Definitions}\label{ss:operators}
Following the idea of \cite{Zel86} on the hyperbolic plane, we will use the Fourier-Helgason transform to globally define pseudo-differential operators on the tree.

\begin{defn}
 Let $a : \T \times \Omega \times [0,\tau] \rightarrow \C$ be a measurable function. We associate an operator $\Op(a)$ with $a$ in the following way: 
 \[ \Op(a)u(x) = \sum_{y\in \T} \int_\Omega \int_0^\tau q^{\left(\frac{1}{2}+is \right)(h_\omega(y) - h_\omega(x))} a(x,\omega,s)u(y) d\nu_x(\omega) d\mu(s)\]
 for every $u : \T \rightarrow \C$ with finite support.
 
 For $c : \T \times \T \times \Omega \times [0,\tau] \rightarrow \C$ measurable, we define the operator $\Opd(c)$ associated with $c$ by
  \[ \Opd(c)u(x) = \sum_{y\in \T} \int_\Omega \int_0^\tau q^{\left(\frac{1}{2}+is \right)(h_\omega(y) - h_\omega(x))} c(x,y,\omega,s)u(y) d\nu_x(\omega) d\mu(s)\]
 for every $u : \T \rightarrow \C$ with finite support.
\end{defn}

Note that if $c(x,y,\omega,s) = c(x,\omega,s)$ does not depend on $y$, then $\Opd(c) = \Op(c)$. The operators of the second definition are therefore more general.

The kernels associated with these operators are easy to extract from the definition. We will denote by 
\[ k_a(x,y) = \int_\Omega \int_0^\tau q^{\left(\frac{1}{2}+is \right)(h_\omega(y) - h_\omega(x))} a(x,\omega,s) d\nu_x(\omega) d\mu(s) \]
the kernel of $\Op(a)$ and by
\[ K_c(x,y) = \int_\Omega \int_0^\tau q^{\left(\frac{1}{2}+is \right)(h_\omega(y) - h_\omega(x))} c(x,y,\omega,s) d\nu_x(\omega) d\mu(s) \]
the kernel of $\Opd(c)$.

These definitions have the important property that they do not depend on the choice of reference point on the tree.

\paragraph{}
We now define classes of functions called \emph{symbols} for which these operators will have interesting properties. The conditions are chosen so that these classes contain interesting examples of symbols (see Example \ref{ex:main}). Let us first consider a general class of \emph{double symbols}

\begin{defn}\label{d:double}
 Let $\S(\T \times \T)$ be the set of functions \[c : \T \times \T \times \Omega \times [0, \tau] \rightarrow \C \] satisfying the following conditions:

\begin{enumerate}
\item\label{double reg omega} For every $x,y \in \T$, $\partial_s^k c(x,y,\omega,s) \in C(\T \times \T \times \Omega \times [0,\tau])$, and for every $l \in \N$ there exists a constant $C_{l}$ such that
\[ \forall n \in \N, \forall x,y \in \T \quad \| (c - \mathcal{E}^x_n c)(x,y, \cdot, \cdot) \|_\infty \leq \frac{C_{l}}{(1+n)^l}. \]

\item For every $k \in \N$, $x,y \in \T$ and $ \omega \in \Omega$, 
\[\partial_s^k c(x,y,\omega, 0) =  \partial_s^k c(x,y, \omega, \tau) = 0.\]
\end{enumerate}
\end{defn}

There is an important subclass of $\S(\T \times \T)$.
\begin{defn}\label{d:symbol}
Let $\epsilon > 0$. We define the class of \emph{semi-classical symbols} $\Sc(\T)$ to be the set of functions ${a_\eps : \T \times \Omega \times [0, \tau] \rightarrow \C}$ satisfying the following conditions:

\begin{enumerate}
\item\label{reg_omega} For every $x \in \T$ and $k \in \N$, $\partial_s^k a_\eps(x,\omega,s) \in C(\Omega \times [0,\tau])$, and for every $l\in \N$ there exists a constant $C_{l}$ such that
\[ \forall n \in \N, \forall x \in \T \quad \| (a_\eps - \mathcal{E}^x_n a_\eps)(x, \cdot, \cdot) \|_\infty \leq \frac{C_{l}}{(1+n)^l}. \]

\item\label{bord} For every $k \in \N$, $x \in \T$ and $ \omega \in \Omega$, $\partial_s^k a_\eps(x,\omega, 0) =  \partial_s^k a_\eps(x, \omega, \tau) = 0$.

\item\label{varx}
\begin{enumerate}
\item\label{varx lipschitz} We have a control over the variation in the first argument of the derivatives of $a_\eps$ with respect to $s$: There exists $C > 0$ such that
\[ \forall x,y \in \T, \forall k \in \N \quad |\partial_s^ k a_\eps(x,\omega,s) - \partial_s^k a_\eps(y,\omega,s)| \leq C \epsilon d(x,y), \]
\item\label{varx cross derivative}For every $l \in \N$, there exists a positive function $t \mapsto C_l(t)$ such that for every $x,y \in \T$ and $n \in \N$
\[ | (a_\eps - \mathcal{E}^x_n a_\eps)(x,\omega,s) - (a_\eps - \mathcal{E}^x_n a_\eps)(y,\omega,s) | \leq \epsilon  \frac{C_l(d(x,y))}{(1+n)^l}. \]
The function $C_l$ is arbitrary. In particular we can think of it as an increasing function.
\end{enumerate}
\end{enumerate}
\end{defn}

\begin{rem}
Condition \ref{reg_omega} in both definitions is a condition of regularity with respect to $\omega$. Condition \ref{bord} is here to compensate for the absence of symmetry condition (see the paragraph ``About the symmetry condition'' in section \ref{s:harmonic}). These two conditions are sufficient to get a rapid decay of the kernels and continuity of the operators on $L^2$. That is why these two results are true for symbols that are only in $\S$. Note that in these two conditions, no regularity with respect to $x$ is asked.

Condition \ref{varx lipschitz} asks that the symbols do not vary too much in $x$. Condition \ref{varx cross derivative} is in a way a ``cross derivative'' between $x$ and $\omega$, the derivative in $x$ being taken before the derivative in $\omega$. The dependence on $d(x,y)$ of the constant reflects the fact that we have the term $\mathcal{E}^x_n a_\eps(y,\cdot,\cdot)$ where the projection $\mathcal{E}^x_n$ is not centered on the first coordinate of $a_\eps$.
\end{rem}

\begin{rem}
Functions $a(x)$ depending only on $x$ do not belong to the class $\S$, but it can be checked that all the results of this article applying to symbols in $\S$ also apply to these functions. This is essentially because the associated operator $\Opd(a)$ (the operator of multiplication by $a$) has a kernel $k_a(x,y)$ vanishing when $x\neq y$, thus satisfying proposition \ref{rapid decay proposition} about the rapid decay of the kernel. We can therefore add these functions to the class without changing the statements of the results. In the same way we can add functions $a(x)$ satisfying only the lipschitz condition \eqref{varx lipschitz} to the class $\Sc$.
\end{rem}

\begin{defn}
Let $\epsilon > 0$. We will call $\epsilon$-\emph{negligible} an operator $A$ with a kernel $K_A(x,y)$ satisfying the following property: for every $N \in \N$ there exists a function $C_N(\epsilon)$ such that
 \begin{equation}
  |K_A(x,y)| \leq C_N(\epsilon) \frac{q^{-\frac{d(x,y)}{2}}}{(1+d(x,y))^N},
 \end{equation}
 and $C_N(\epsilon) \rightarrow 0$ when $\epsilon \rightarrow 0$.
 
 The set $\Psi$ of \emph{pseudo-differential operators} is defined as the set of operators $\Opd(c)$ associated with symbols $c \in \S(\T \times \T)$, modulo $\epsilon$-negligible operators.
\end{defn}

\subsection{Properties and examples}
The classes of symbols of definitions \ref{d:double} and \ref{d:symbol} are algebras over $\C$. Indeed the main difficulty is to prove the following proposition.
\begin{prop}\label{product closure}
 If $a,b \in \S$ then $ab \in \S$, and if $a,b \in \Sc$, then $ab \in \Sc$.
\end{prop}
\begin{proof}
 We will only prove the proposition for $a,b \in \Sc$, it will be clear that the case $a,b \in \S$ can be treated in a similar way. 
 
 The following lemma will be used several times throughout the proof
\begin{lem}\label{sublemma : product closure}
 For every function $a : \T \times \Omega \times [0,\tau] \rightarrow \C$ and $b : \T \times \Omega \times [0,\tau] \rightarrow \C$, 
 \[ \left| \mathcal{E}^x_n a  \mathcal{E}^x_n b -\mathcal{E}^x_n (ab))(y,\omega,s) \right|  \leq \sup_{\omega'} a(y,\omega',s) \sup_{\omega'} \left| \mathcal{E}^x_n b(y,\omega',s) - b(y,\omega',s) \right| \]
 for all $x,y \in \T$, and $(\omega,s) \in \Omega \times [0,\tau]$.
\end{lem}
\begin{proof}
 To simplify the notation, we can ignore the dependence in $x,y$ and $s$.
 
 Recall from $\eqref{e:mean operator}$ that
  \[ \mathcal{E}_n (ab)(\omega_0)  = \frac{1}{\nu(\Omega_n(o,\omega_0))} \int_{\Omega_n(o,\omega_0)} (ab)(\omega) d\nu(\omega). \]
 We also have
 \[ \mathcal{E}_n a(\omega_0) \mathcal{E}_n b(\omega_0)  = \frac{1}{\nu(\Omega_n(o,\omega_0))^2} \int_{\Omega_n(o,\omega_0)} \int_{\Omega_n(o,\omega_0)} a(\omega)b(\omega') d\nu(\omega)d\nu(\omega').\]
Let us write $dm(\omega) = \frac{d\nu(\omega)}{\nu(\Omega_n(o,\omega_0))}$, which gives a probability measure on $\Omega_n(o,\omega_0)$. We have
\begin{align*}
(\mathcal{E}_n a  \mathcal{E}_n b -\mathcal{E}_n (ab))(\omega_0)
  &=\int_{\Omega_n(o,\omega_0)} \int_{\Omega_n(o,\omega_0)} \left( a(\omega)b(\omega') - a(\omega)b(\omega) \right) dm(\omega')dm(\omega) \\
  &=\int_{\Omega_n(o,\omega_0)} \int_{\Omega_n(o,\omega_0)} a(\omega) ( b(\omega') - b(\omega) ) dm(\omega')dm(\omega) \\
  &=\int_{\Omega_n(o,\omega_0)}  a(\omega) \left( \int_{\Omega_n(o,\omega_0)} b(\omega') dm(\omega') - b(\omega) \right) dm(\omega) \\
  &=\int_{\Omega_n(o,\omega_0)}  a(\omega) \left( \mathcal{E}_n b(\omega_0) - b(\omega) \right) dm(\omega),
\end{align*}
but $\mathcal{E}_n b(\omega_0) = \mathcal{E}_n b(\omega)$ for all $\omega \in \Omega_n(o,\omega_0)$ by definition. So
\begin{align*}
 \left| (\mathcal{E}_n a  \mathcal{E}_n b -\mathcal{E}_n (ab))(\omega_0) \right| 
 & = \left| \int_{\Omega_n(o,\omega_0)}  a(\omega) \left( \mathcal{E}_n b(\omega) - b(\omega) \right) dm(\omega)\right| \\
 & \leq \sup_\omega a(\omega) \sup_\omega \left| \mathcal{E}_n b(\omega) - b(\omega) \right|.
\end{align*}
\end{proof}
Let us now begin the proof of proposition \ref{product closure}.
\begin{enumerate}
  \item  As $x$ and $s$ are fixed, we ignore the dependence in these variables to simplify the notation. We will put it back at the end. We decompose
 \[ ab - \mathcal{E}_n (ab) = a(b - \mathcal{E}_n b) + \mathcal{E}_n b(a - \mathcal{E}_n a) + ( (\mathcal{E}_n a) (\mathcal{E}_n b) -\mathcal{E}_n (ab)). \]
 The first two terms are easy to bound, we have
  \[ |ab - \mathcal{E}_n (ab)| \leq  \|a\|_\infty |b - \mathcal{E}_n b| + \|b\|_\infty |a - \mathcal{E}_n a| + |(\mathcal{E}_n a) (\mathcal{E}_n b) -\mathcal{E}_n (ab)|, \]
  because $ \|\mathcal{E}_n b\|_\infty \leq \|b\|_\infty$. For the last term, we use lemma \ref{sublemma : product closure}, which gives
   \[ \left| \mathcal{E}_n a  \mathcal{E}_n b -\mathcal{E}_n (ab))(\omega) \right|  
   \leq \sup_{\omega'} a(\omega') \sup_{\omega'} \left| \mathcal{E}_n b(\omega') - b(\omega') \right|. \]
We thus have finally, for all $(x,\omega,s)$
\[ |ab - \mathcal{E}^x_n (ab)|(x,\omega,s) \leq 2\|a\|_\infty \sup_{\omega'} \left| b - \mathcal{E}^x_n b \right|(x,\omega',s) + \|b\|_\infty \sup_{\omega'} |a - \mathcal{E}^x_n a|(x,\omega',s). \]
As $a$ and $b$ satisfy the inequality of condition $\ref{reg_omega}$, this proves that the product $ab$ satisfies it too. We still have to prove the inequality for the derivatives $\partial_s^k(ab)$. But for all $k \in \N$, $\partial_s^k(ab)$ is a linear combination of products of the form $\partial_s^i a \partial_s^{k-i} b$ with $i \in \{0,\ldots,k\}$ and each factor satisfies the inequality. The preceding proof thus gives the result.
 \item The second condition is clear if we write $\partial_s^k(ab)$ as a linear combination of products of the form $\partial_s^i a \partial_s^{k-i} b$ with $i \in \{0,\ldots,k\}$.
 \item We use the fact that, ignoring the dependence in $\omega$ and $s$, 
 \[ (ab)(x) - (ab)(y) = a(x)(b(x) - b(y)) + b(y)(a(x) - a(y)).\] 
 We then apply this to the derivatives of $(ab)(x,\omega,s)$ with respect to $s$ to obtain condition \ref{varx lipschitz}.

 To obtain condition \ref{varx cross derivative}, we will combine some of the preceding ideas. Let us ignore the dependence in $s$ and work only with $x$ and $\omega$.
 \begin{multline}\label{main expression product cross derivative}
  ((ab)(x,\omega) - (ab)(y,\omega)) - \mathcal{E}_n^x ((ab)(x,\omega) - (ab)(y,\omega)) \\
  = a (x,\omega)(b(x,\omega) - b(y,\omega)) + b(y,\omega)(a(x,\omega)- a(y,\omega)) \\
  - \mathcal{E}_n^x(a(x,\omega)(b(x,\omega) - b(y,\omega))) - \mathcal{E}_n^x (b(y,\omega)(a(x,\omega)- a(y,\omega)))
 \end{multline}
According to lemma \ref{sublemma : product closure}, we have
\[ \mathcal{E}_n^x(a(x,\omega)(b(x,\omega) - b(y,\omega))) =  \mathcal{E}_n^xa(x,\omega)\mathcal{E}_n^x(b(x,\omega) - b(y,\omega)) + R(x,y,\omega)\]
with
\[ |R(x,y,\omega)| \leq \sup_{\omega'} (b(x,\omega')- b(y,\omega')) \sup_{\omega'}(\mathcal{E}_n^x a(x,\omega') - a(x,\omega')).\]
 Using condition \ref{reg_omega} for $b$ and condition \ref{varx lipschitz} for $a$ we have
\begin{equation}\label{inequality R}
 |R(x,y,\omega)| \leq C_l \epsilon \frac{d(x,y)}{(1+n)^l}. 
\end{equation}

Still using lemma \ref{sublemma : product closure},
\[ \mathcal{E}_n^x(b(y,\omega)(a(x,\omega) - a(y,\omega))) =  \mathcal{E}_n^xb(y,\omega)\mathcal{E}_n^x(a(x,\omega) - a(y,\omega)) + R'(x,y,\omega)\]
where
\[ |R'(x,y,\omega)| \leq \sup_{\omega'}(a(x,\omega')- a(y,\omega')) \sup_{\omega'}(\mathcal{E}_n^x b(y,\omega') - b(y,\omega')). \]
We have for all $\omega'$ 
\begin{multline*}
 \mathcal{E}_n^x b(y,\omega') - b(y,\omega') \\
 = \mathcal{E}_n^x(b(y,\omega') - b(x,\omega')) - (b(y,\omega') - b(x,\omega')) + \mathcal{E}_n^x b(x,\omega') - b(x,\omega')
\end{multline*}
so according to conditions \ref{varx cross derivative} and \ref{reg_omega}:
\begin{equation}\label{reg omega xy}
  |  \mathcal{E}_n^x b(y,\omega') - b(y,\omega') | \leq \frac{1}{(1+n)^l} \big( C_l + \epsilon C_l(d(x,y)) \big)
\end{equation}
and using condition \ref{varx lipschitz} we finally have
\begin{align}
 |R'(x,y,\omega)| &\leq C_l \frac{1}{(1+n)^l} \big( \epsilon d(x,y) + 2\|a\|_\infty \epsilon C_l(d(x,y)) \big) \nonumber \\
 &\leq C_{a,l} \epsilon \frac{C'_l(d(x,y))}{(1+n)^l}. \label{inequality R'}
\end{align}
where $C'_l(t) \geq \max\{t,C_l(t)\}$ for all $x,y \in \T$.
Going back to the main expression $\eqref{main expression product cross derivative}$ we obtain after some more modifications
 \begin{align*}
  ((ab)&(x,\omega) - (ab)(y,\omega)) - \mathcal{E}_n^x ((ab)(x,\omega) - (ab)(y,\omega)) \\
  = &a(x,\omega) ( b(x,\omega) - b(y,\omega) - \mathcal{E}_n^x (b(x,\omega) - b(y,\omega)) \\
  & \quad + \mathcal{E}_n^x(b(x,\omega) - b(y,\omega))(a(x,\omega) - \mathcal{E}_n^xa(x,\omega)) \\
  & \qquad + R(x,y,\omega) \\
  &+ b(y,\omega)(a(x,\omega) - a(y,\omega) - \mathcal{E}_n^x(a(x,\omega) -a(y,\omega))) \\
  & \quad + \mathcal{E}_n^x(a(x,\omega) - a(y,\omega))(b(y,\omega) - \mathcal{E}_n^x b(y,\omega)) \\
  & \qquad + R'(x,y,\omega).
 \end{align*}
Using the different conditions in the symbol class and inequality $\eqref{reg omega xy}$ in addition to the inequalities $\eqref{inequality R}$ and $\eqref{inequality R'}$ on $R$ and $R'$ we obtain
\begin{align*}
   ((ab)&(x,\omega) - (ab)(y,\omega)) - \mathcal{E}_n^x ((ab)(x,\omega) - (ab)(y,\omega)) \\
   \leq &\|a\|_\infty C_l \epsilon \frac{C_l(d(x,y))}{(1+n)^l} \\
   & \quad + C \epsilon d(x,y)  \frac{C_l}{(1+n)^l}\\
   & \qquad + C_l \epsilon \frac{d(x,y)}{(1+n)^l} \\
   & + \|b\|_\infty C_l \epsilon \frac{C_l(d(x,y))}{(1+n)^l} \\
   & \quad + \frac{1}{(1+n)^l} ( C_l \epsilon d(x,y) + 2 \|a\|_\infty \epsilon C_l(d(x,y))) \\
   & \qquad + C_{a,l} \epsilon \frac{C'_l(d(x,y))}{(1+n)^l}\\
\end{align*}
and finally
\begin{equation*}
 ((ab)(x,\omega) - (ab)(y,\omega)) - \mathcal{E}_n^x ((ab)(x,\omega) - (ab)(y,\omega)) \leq \epsilon \frac{C''_l(d(x,y))}{(1+n)^l},
\end{equation*}
where $t \mapsto C''_l(t)$ is an increasing function such that for all $x,y \in \T$, 
\[ C''_l(d(x,y)) \geq \max\{d(x,y), C_l(d(x,y))\}. \]
This gives us condition \ref{varx cross derivative}.
\end{enumerate}
\end{proof}

We will now show that the symbol classes are also closed under the action of operators related to the dynamics on the tree. These properties are important for quantum ergodicity. They also allow us to give more elaborate examples of symbols. First, we need some definitions.

\begin{defn}\label{shift operator}
 Recall that any point $(x,\omega) \in \T \times \Omega$ can be written as a half-geodesic $(x,x_1,x_2,\ldots)$ starting at $x$. We define the \emph{shift} $\sigma : \T \times \Omega \rightarrow \T \times \Omega$
acting on these sequences
 \[ \sigma(x,\omega) = \sigma(x,x_1,x_2, \ldots) = (x_1,x_2, \ldots) = (x_1,\omega). \]
 We will also denote by $\sigma$ the map defined on $\T \times \Omega \times [0,\tau]$  such that $\sigma(x,\omega,s) = (x_1,\omega,s)$, and write $\sigma_\omega(x) = x_1$.
 
The map $\sigma$ has a left inverse, the \emph{transfer operator} $L$ defined on functions $a(x,\omega,s)$ by
 \[ La(x,\omega,s) = \frac{1}{q} \sum_{y: \sigma_\omega(y) = x} a(y,\omega,s). \]
\end{defn}

\begin{prop}\label{prop stab geodesic}
 Let $a \in \Sc$ be a semi-classical symbol. Its composition with the shift is still a semi-classical symbol, that is $a\circ \sigma \in \Sc$.
\end{prop}
\begin{rem}
 It will be clear from the following proof that $\S$ is also closed under composition by the shift. 
\end{rem}
\begin{proof}
 As $\sigma$ does not act on the variable $s$, condition \ref{bord} of definition \ref{d:symbol} is satisfied by $a \circ \sigma$. 
 
 Let us now ignore the dependence in $s$, fix $\omega \in \Omega$, $x\in\T$ and write $\sigma(x,\omega) = (x_1,\omega)$.
 When $n \geq 2$, 
 \[ |a \circ \sigma(x,\omega) - \mathcal{E}_n^x (a \circ \sigma)(x,\omega)| = |a(x_1,\omega) - \mathcal{E}_{n-1}^{x_1} a(x_1,\omega)| \leq C_l \frac{1}{n^l} \leq 2C_l \frac{1}{(1+n)^l}\]
 so condition \ref{reg_omega} is also satisfied by $a\circ\sigma$.
 
 Now fix $y\in\T$ and write $\sigma(y,\omega) = (y_1,\omega)$. Notice that $d(x_1,y_1) \leq d(x,y)$, so
 \[ |a\circ\sigma(x,\omega) - a\circ\sigma(y,\omega)| = |a(x_1,\omega) - a(y_1,\omega)| \leq C \epsilon d(x_1,y_1) \leq C \epsilon d(x,y)  \]
 and this is also true for the derivatives of $a$ with respect to $s$ when we add back the dependence in this variable, so condition \ref{varx lipschitz} is satisfied. 
 
 Let us now prove that condition \ref{varx cross derivative} is satisfied. That is: for every $l,n \in \N$, $x,y \in \T$ and $\omega \in \Omega$, 
\[  |a\circ\sigma(x,\omega) - a\circ\sigma(y,\omega) - \mathcal{E}_n^x( a\circ\sigma(x,\omega) - a\circ\sigma(y,\omega))| \leq \frac{C_l(d(x,y))}{(1+n)^l} \epsilon, \] for some function $t \mapsto C_l(t)$.
 
 We have for all $n \geq 2$
 \[ \mathcal{E}_n^x( a\circ\sigma(x,\omega)) = \mathcal{E}_{n-1}^{x_1}( a(x_1,\omega)).\]
 Recall now from $\eqref{e:mean operator}$ that
\[ \mathcal{E}_n^x (a \circ \sigma) (y,\omega) = \frac{1}{\nu_x(\Omega_n(x,\omega))} \int_{\Omega_n(x,\omega)} a \circ \sigma(y,\omega') d\nu_x(\omega'), \]
where $\Omega_n(x,\omega)$ is the cylinder starting at $x$ of length $n$ in the direction of $\omega$.
 If $n \geq d(x,y) + 1$ then for all $\omega' \in \Omega_n(x,\omega)$, $\sigma(y,\omega') = (y_1,\omega')$, thus
 \[ \mathcal{E}_n^x(a\circ\sigma(y,\omega)) = \mathcal{E}_{n-1}^{x_1}(a(y_1,\omega)).\]
 Therefore
 \[ \mathcal{E}_n^x( a\circ\sigma(x,\omega) - a\circ\sigma(y,\omega)) 
    = \mathcal{E}_{n-1}^{x_1}( a(x_1,\omega) - a(y_1,\omega)).\]
 Using the fact that $a \in \Sc$ satisfies condition \ref{varx cross derivative}, we have
\begin{align*}
 |a\circ\sigma & (x,\omega) - a\circ\sigma(y,\omega) - \mathcal{E}_n^x( a\circ\sigma(x,\omega) - a\circ\sigma(y,\omega))| \\
    &= |a(x_1,\omega) - a(y_1,\omega) -\mathcal{E}_{n-1}^{x_1}( a(x_1,\omega) - a(y_1,\omega))| \\
    &\leq \frac{C_l(d(x_1,y_1))}{n^l} \epsilon.
\end{align*}
Therefore we have 
\begin{equation}\label{eq stab geodesic 1}
  |a\circ\sigma(x,\omega) - a\circ\sigma(y,\omega) - \mathcal{E}_n^x( a\circ\sigma(x,\omega) - a\circ\sigma(y,\omega))| \leq \frac{C'_l(d(x,y))}{(1+n)^l} \epsilon,
\end{equation}
 and condition \ref{varx cross derivative} is thus satisfied when $n \geq 1 + d(x,y)$.

 Let us now consider the case $n \leq 1 + d(x,y)$. Recall that
 \[\mathcal{E}_n^x (a\circ\sigma(x,\omega) - a\circ\sigma(y,\omega)) = \frac{1}{\nu(\Omega_n(x,\omega))} \int_{\Omega_n(x,\omega)} a\circ\sigma(x,\omega') - a\circ\sigma(y,\omega') d\nu_x(\omega'). \]
 We have to estimate
 \begin{multline}\label{eq stab geodesic 2}
  a\circ\sigma(x,\omega) - a\circ\sigma(y,\omega) - \mathcal{E}_n^x( a\circ\sigma(x,\omega) - a\circ\sigma(y,\omega)) \\
  = \frac{1}{\nu(\Omega_n(x,\omega))} \int_{\Omega_n(x,\omega)} a\circ\sigma(x,\omega) - a\circ\sigma(x,\omega') - (a\circ\sigma(y,\omega) - a\circ\sigma(y,\omega')) d\nu_x(\omega').
 \end{multline}
Note that $a\circ\sigma(x,\omega) = a(x_1,\omega)$ and $a\circ\sigma(x,\omega') = a(x'_1,\omega')$ with $d(x_1,x'_1) \leq 2$.  We thus have
\[ |a\circ\sigma(x,\omega) - a\circ\sigma(x,\omega')| \leq 2C \epsilon\]
and the same is true when we replace $x$ with $y$. Therefore \eqref{eq stab geodesic 2} gives
\[ |a\circ\sigma(x,\omega) - a\circ\sigma(y,\omega) - \mathcal{E}_n^x( a\circ\sigma(x,\omega) - a\circ\sigma(y,\omega))| \leq 4C \epsilon.  \]
Moreover, as $n \leq 1 + d(x,y)$ we have $ (1+n) \leq (2+d(x,y))$. Thus for every $l \in \N$
\[ (1+n)^l|a\circ\sigma(x,\omega) - a\circ\sigma(y,\omega) - \mathcal{E}_n^x( a\circ\sigma(x,\omega) - a\circ\sigma(y,\omega))| \leq (2+d(x,y))^l 4 C \epsilon,  \]
in this case.
Together with $\eqref{eq stab geodesic 1}$ it gives for every $l,n \in \N$, $x,y \in \T$ and $\omega \in \Omega$, 
\[  |a\circ\sigma(x,\omega) - a\circ\sigma(y,\omega) - \mathcal{E}_n^x( a\circ\sigma(x,\omega) - a\circ\sigma(y,\omega))| \leq \frac{C_l(d(x,y))}{(1+n)^l} \epsilon, \] for some function $C_l(d(x,y))$ of $d(x,y)$ that grows at least like $d(x,y)^l$. This is compatible with condition \ref{varx cross derivative}.
\end{proof}

\begin{rem}
 From the preceding proof, we see that if $a \in \Sc$, then the symbol $a\circ\sigma \in \Sc$ will satisfy condition \ref{varx cross derivative} of definition \ref{d:symbol}  with for all $l\in \N$, $C_l(t) \geq t^l$. 
\end{rem}

\paragraph*{}
A similar proposition can be proved for the transfer operator $L$.
\begin{prop}\label{prop stab transfer}
 Let $a \in \Sc$ be a semi-classical symbol. Then $La \in \Sc$.
\end{prop}
\begin{rem}
 As in proposition \ref{prop stab geodesic}, $\S$ is also closed under the action of $L$. 
\end{rem}
\begin{proof}
 We follow the same steps as in the proof of proposition \ref{product closure}. But here, $L$ gives us more regularity in $\omega$ and less in $x$.
 
 Recall from definition \ref{shift operator} that
 \[ La(x,\omega,s) = \frac{1}{q} \sum_{x':\sigma_\omega(x') = x} a(x',\omega,s). \]
 As $L$ does not act on the variable $s$, condition \ref{bord} of definition \ref{d:symbol} is satisfied by $La$. Let us now ignore the dependence in $s$, fix $\omega \in \Omega$ and $x\in\T$, and prove that condition \ref{reg_omega} is satisfied.
 
 Using the definition $\eqref{e:mean operator}$ of the operator $\mathcal{E}^x_n$ on $La$ we have
\[ \mathcal{E}_n^x (La) (x,\omega) = \frac{1}{\nu_x(\Omega_n(x,\omega))} \frac{1}{q}  \int_{\Omega_n(x,\omega)} \sum_{x':\sigma_{\omega'}(x') = x}  a(x',\omega') d\nu_x(\omega'). \]
If $n \geq 1$, then $\sigma_{\omega'}(x') = \sigma_\omega(x')$ for every $\omega' \in \Omega_n(x,\omega)$. So we have
\begin{align}
  \mathcal{E}_n^x (La)(x,\omega) &= \frac{1}{q} \sum_{x':\sigma_{\omega}(x') = x} \frac{1}{\nu_x(\Omega_n(x,\omega))} \int_{\Omega_n(x,\omega)} a(x',\omega') d\nu_x(\omega') \nonumber \\
  &= \frac{1}{q} \sum_{x':\sigma_\omega(x') = x} \mathcal{E}_n^x a(x',\omega) \label{eq stab transfer 1} \\
  &= \frac{1}{q} \sum_{x':\sigma_\omega(x') = x} \mathcal{E}_{n+1}^{x'} a(x',\omega) \nonumber
\end{align}
 which can be used to bound the following expression 
 \begin{align*}
  |La(x,\omega) - \mathcal{E}_n^x (La)(x,\omega)| &= \left| \frac{1}{q} \sum_{x':\sigma_\omega(x') = x}  (a(x',\omega) - \mathcal{E}_{n+1}^{x'} a(x',\omega)) \right| \\
  &\leq C_l \frac{1}{(2+n)^l} \leq C_l \frac{1}{(1+n)^l},
 \end{align*}
 because $a$ satisfies condition \ref{reg_omega}. Therefore condition \ref{reg_omega} is also satisfied by $La$.
 
 For condition \ref{varx lipschitz}, fix $y\in\T$. If $x',y' \in \T$ are such that $\sigma_\omega(y')=y$ and $\sigma_\omega(x') = x$, then $d(x',y') \leq d(x,y)+2$. So 
 \begin{align}
  |La(x,\omega) - La(y,\omega)| &= \frac{1}{q} \left|\sum_{x':\sigma_\omega(x') = x} a(x',\omega) - \sum_{y':\sigma_\omega(y') = y} a(y',\omega) \right| \nonumber\\
  & \leq C \epsilon (d(x,y) +2) \label{eq stab transfer 2}
 \end{align}
 in whatever way we regroup the terms of the two sums in $q$ differences, and using the fact that $a$ satisfies condition \ref{varx lipschitz}. Notice then that if $d(x,y) \geq 1$ then we get from $\eqref{eq stab transfer 2}$ that \[ |La(x,\omega) - La(y,\omega)| \leq 3C \epsilon d(x,y),\] and this inequality is also true when $d(x,y) =0$, so it is true for all $x,y \in \T$  and condition \ref{varx lipschitz} is satisfied. 
 
  Let us now prove that condition \ref{varx cross derivative} is satisfied. That is for every $l,n \in \N$, $x,y \in \T$ and $\omega \in \Omega$, 
\[  |La(x,\omega) - La(y,\omega) - \mathcal{E}_n^x( La(x,\omega) - La(y,\omega))| \leq \frac{C_l(d(x,y))}{(1+n)^l} \epsilon, \] for some function $t \mapsto C_l(t)$. 

Recall that
\[ \mathcal{E}_n^x (La) (y,\omega) = \frac{1}{\nu_x(\Omega_n(x,\omega))} \frac{1}{q}  \int_{\Omega_n(x,\omega)} \sum_{y':\sigma_{\omega'}(y') = y}  a(y',\omega') d\nu_x(\omega'). \]
 If $n \geq d(x,y) + 1$ then $\{y' : \sigma_{\omega'}(y') = y \} = \{y' : \sigma_{\omega}(y') = y \}$ for every $\omega' \in \Omega_n(x,\omega)$. So
 \begin{align*}
  \mathcal{E}_n^x (La) (y,\omega) &= \frac{1}{q}  \sum_{y':\sigma_{\omega}(y') = y}  \frac{1}{\nu_x(\Omega_n(x,\omega))}  \int_{\Omega_n(x,\omega)} a(y',\omega') d\nu_x(\omega') \\
  &= \frac{1}{q}  \sum_{y':\sigma_{\omega}(y') = y} \mathcal{E}_{n}^x a(y',\omega),
 \end{align*}
 and recall from equality $\eqref{eq stab transfer 1}$ at the beginning of the proof that 
 \[\mathcal{E}_n^x (La)(x,\omega)= \frac{1}{q} \sum_{x':\sigma_\omega(x') = x} \mathcal{E}_{n}^{x} a(x',\omega)\]
For every $x',y' \in \T$ such that $\sigma_\omega(x') = x$ and $\sigma_\omega(y')= y$, we have
\begin{align*}
 |a(x',&\omega) - a(y',\omega) - \mathcal{E}_n^x( a(x',\omega) - a(y',\omega))| \\
    &= |a(x',\omega) - a(y',\omega) -\mathcal{E}_{n+1}^{x'}( a(x',\omega) - a(y',\omega))| \\
    &\leq \frac{C_l(d(x',y'))}{(2+n)^l} \epsilon,
\end{align*}
 using for the last inequality the fact that $a \in \Sc$ satisfies condition \ref{varx cross derivative}.
Therefore when we sum over pairs $(x',y')$ formed by grouping arbitrarily the elements of the two sets $\{x' : \sigma_{\omega}(x') = x \} $ and $\{y' : \sigma_{\omega}(y') = y \} $, and divide by $q$, we obtain 
\begin{equation}\label{eq stab transfer 3}
  |La(x,\omega) - La(y,\omega) - \mathcal{E}_n^x( La(x,\omega) - La(y,\omega))| \leq \frac{C'_l(d(x,y))}{(1+n)^l} \epsilon,
\end{equation}
 and condition \ref{varx cross derivative} is thus satisfied when $n \geq 1 + d(x,y)$.

 Let us now consider the case $n \leq 1 + d(x,y)$. The end of the proof is the same as that of proposition \ref{prop stab geodesic}. We just have to replace the composition with the shift with the transfer operator. Recall that for each $n \leq 1 + d(x,y)$,
 \[\mathcal{E}_n^x (La(x,\omega) - La(y,\omega)) = \frac{1}{\nu(\Omega_n(x,\omega))} \int_{\Omega_n(x,\omega)} La(x,\omega') - La(y,\omega') d\nu_x(\omega'). \]
 We have to estimate
 \begin{multline}\label{eq stab transfer 4}
  La(x,\omega) - La(y,\omega) - \mathcal{E}_n^x( La(x,\omega) - La(y,\omega)) \\
  = \frac{1}{\nu(\Omega_n(x,\omega))} \int_{\Omega_n(x,\omega)} La(x,\omega) - La(x,\omega') - (La(y,\omega) - La(y,\omega')) d\nu_x(\omega').
 \end{multline}
Recall that 
\[ La(x,\omega) = \frac{1}{q} \sum_{x':\sigma_\omega(x') = x} a(x',\omega) \quad \text{ and } \quad La(x,\omega') = \frac{1}{q} \sum_{x'':\sigma_{\omega'}(x'') = x} a(x'',\omega'). \]
For each $x',x'' \in \T$ such that $\sigma_{\omega}(x') = x$ and $\sigma_{\omega'}(x'') = x$, we have $d(x',x'') \leq 2$.  We thus have
\[ |La(x,\omega) - La(x,\omega')| \leq 2C \epsilon\]
and the same is true when we replace $x$ with $y$. Therefore \eqref{eq stab transfer 4} gives
\[ |La(x,\omega) - La(y,\omega) - \mathcal{E}_n^x( La(x,\omega) - La(y,\omega))| \leq 4C \epsilon.  \]
Moreover, as $n \leq 1 + d(x,y)$ we have $ (1+n) \leq (2+d(x,y))$. Thus for every $l \in \N$
\[ (1+n)^l| La(x,\omega) - La(y,\omega) - \mathcal{E}_n^x( La(x,\omega) - La(y,\omega))| \leq (2+d(x,y))^l 4 C \epsilon,  \]
in this case.
Together with $\eqref{eq stab transfer 3}$ it gives for every $l,n \in \N$, $x,y \in \T$ and $\omega \in \Omega$, 
\[  |La(x,\omega) - La(y,\omega) - \mathcal{E}_n^x( La(x,\omega) - La(y,\omega))| \leq \frac{C_l(d(x,y))}{(1+n)^l} \epsilon, \] for some function $C_l(d(x,y))$ of $d(x,y)$ that grows at least like $d(x,y)^l$. This is compatible with condition \ref{varx cross derivative}.
\end{proof}

 Using the previous properties, we will now give some natural and nontrivial examples of symbols belonging to our symbol classes.

\begin{exmp}\label{ex:main}
 Condition \ref{reg_omega} tells us that the symbol must not vary too much with respect to $\omega$. Recall that a pair $(x,\omega)$ is equivalent to a half-geodesic starting at $x$ : $[x,\omega) = (x, x_1, x_2, \ldots)$. We can therefore write equivalently $a(x,\omega,s)$ or $a(x, x_1, x_2, \ldots,s)$. Asking that the symbol depend only on a finite number of coordinates --- more precisely that there exists $n \in \N$ such that for every $(x_1,\ldots,x_n) \in \T^n$, 
 \[ a(x_1,\ldots, x_n, x_{n+1}, x_{n+2} \ldots,s) = a(x_1,\ldots, x_n, x'_{n+1}, x'_{n+2}, \ldots,s)\] for all half geodesics $(x_n, x_{n+1}, x_{n+2} \ldots)$ and $(x_n, x'_{n+1}, x'_{n+2} \ldots)$ --- is a simple way of satisfying condition \ref{reg_omega}. Indeed, if $a$ depends only on $n_0$ coordinates, then 
 \[ \sup_x \| a(x,\cdot,\cdot) - \mathcal{E}_n^x a(x,\cdot,\cdot) \|_\infty = 0, \]
 for all $n \geq n_0$. Incidentally, condition \ref{varx cross derivative} is then also satisfied. But it is not clear at first sight if this can be easily coupled with condition \ref{varx lipschitz}, asking that the symbols do not vary too much with respect to $x$, for symbols depending on more than one coordinate.
 
 A way to get a symbol satisfying all the conditions and depending on more than one coordinate is to start from a function $a(x,s)$ satisfying conditions \ref{bord} and \ref{varx lipschitz}. As $a$ does not depend on $\omega$, it automatically satisfies condition \ref{reg_omega} and \ref{varx cross derivative} and therefore belongs to $\Sc$. Then for every $k \in \N$, $a\circ \sigma^k$ is still a symbol in $\Sc$, according to proposition \ref{prop stab geodesic}, and it depends on $k$ coordinates. This is a fundamental example of symbol for quantum ergodicity.
\end{exmp}

\section{Rapid decay of the kernel}\label{s:decay}
Many proofs of pseudo-differential calculus theorems will relie on the following result about the rapid decay of the kernel away from the diagonal.

\begin{prop}\label{rapid decay proposition}
The kernel of the operator $\Opd(a)$ associated with a double symbol $a \in \S(\T \times \T)$, defined by
\[ K_a(x,y) = \int_{\Omega} \int_0^\tau q^{(\frac{1}{2}+is)(h_\omega(y) - h_\omega(x))} a(x,y,\omega,s)  d\nu_x(\omega)  d\mu(s), \]
has the following property: for all $N \in \mathbb{N}$ there exists a constant $C_a(N)$ such that
 \[  |K_a(x,y)| \leq C_a(N) \frac{q^{-\frac{d(x,y)}{2}}}{(1+d(x,y))^N}\]
for every  $x,y \in \T$. The dependence on $a$ of the constant $C_a(N)$ is given by
\[ C_a(N) = C_N \left( \|a\|_{\Omega,N} + \sum_{k=0}^{N+1} \|\partial_s^k a \|_\infty \right), \]
where
 \[ \|a\|_{\Omega,N} = \sup_{x,y \in \T} \sup_{n \in \N} (n+1)^N \| (a - \mathcal{E}^x_n a)(x,y, \cdot, \cdot) \|_\infty \]
 is the constant of condition \ref{double reg omega} of definition \ref{d:double} with $l=N$.
\end{prop}

\begin{rem}
The proof is an adaptation of the second part of the proof of Theorem 2 in \cite{CS99}. The regularity in $s$ of the symbol and integration by parts are used to get a decay of the kernel. The main difference with \cite{CS99} is that our variable $s$ is in $[0,\tau]$ rather than $[-\tau,\tau]$ and the \emph{symmetry condition} is replaced by condition \ref{bord} of definition \ref{d:double}.
\end{rem}

\begin{proof}
We fix $x,y \in \T$ and $N \in \N$. As the definition of the kernel does not depend on the choice of a reference point, we can take $x$ to be the new reference point ($x=o$), to simplify the notation.
We thus have to prove that there exists a constant $C_a(N)$ independent of $o$ and $y$, such that the kernel 
\[ K_a(o,y)  = \int_{\Omega} \int_0^\tau q^{(\frac{1}{2}+is)h_\omega(y)} a(o,y,\omega,s) d\nu(\omega) d\mu(s) \]
satisfies 
\[  |K_a(o,y)| \leq C_a(N) \frac{q^{-\frac{|y|}{2}}}{(1+|y|)^N}.\]

We take $M$ to be the integer part of $\frac{|y|}{3}$ and we decompose \[ a(o,y,\omega,s) = (a - \mathcal{E}_M a)(o,y,\omega,s) + \mathcal{E}_M a(o,y,\omega,s). \]
The first part of the decomposition leads to the following computation
\begin{align*}
	&\left| \int_{\Omega} \int_0^\tau q^{(\frac{1}{2}+is)h_\omega(y)} (a - \mathcal{E}_M a)(o,y,\omega,s)  d\nu(\omega)  d\mu(s) \right| \\
	& \quad \leq \|a - \mathcal{E}_M a \|_\infty \left| \sum_{j=0}^{|y|} \int_{E_j(y)} q^{j-\frac{|y|}{2}} d\nu(\omega) \right| \\
	& \quad \leq  \|a - \mathcal{E}_M a \|_\infty \sum_{j=0}^{|y|} \nu(E_j(y)) q^{j-\frac{|y|}{2}} \\
	& \quad \leq  \|a - \mathcal{E}_M a \|_\infty (1+|y|) q^{-\frac{|y|}{2}} \\
	& \quad \leq \|a \|_{0,N+1} \frac{1}{(1+M)^{N+1}} (1+|y|) q^{-\frac{|y|}{2}} \\
	& \quad \leq C_N \|a \|_{0,N+1} \frac{q^{-\frac{|y|}{2}}}{(1+|y|)^N} 
\end{align*}
where we used the fact that $\nu(E_j(y)) \leq q^{-j}$ and that $1+M \geq C(1+|y|)$.

For the second part of the decomposition we denote $A_M(y,s) = \mathcal{E}_M a(o,y, \omega,s)$ for any $\omega \in \cup_{j = M}^{|y|} E_j(y)$ (for every $s$, $\mathcal{E}_M a(o,y, \cdot ,s)$ is constant on this set), and \[ \mathcal{A}_M(y,\omega,s) = \mathcal{E}_M a(o,y, \omega,s) - A_M(y,s), \]
so that we can write
\[ \mathcal{E}_M a(x,y,\omega,s) =  \mathcal{A}_M(y,\omega,s) + A_M(y,s). \] 
Because $|y| \geq M$ we then have
\begin{align*}
 \int_{\Omega} \int_0^\tau & q^{(\frac{1}{2}+is)h_\omega(y)} \mathcal{E}_M a(o,y,\omega,s)  d\nu(\omega)  d\mu(s) \\
	& = \sum_{j=0}^{|y|} \int_{E_j(y)} \int_0^\tau q^{(\frac{1}{2}+is)(2j - |y|)} \mathcal{A}_M(y,\omega,s)  d\nu(\omega)  d\mu(s) \\
	&\qquad + \int_{\Omega} \int_0^\tau  q^{(\frac{1}{2}+is) h_\omega(y)} A_M(y,s)  d\nu(\omega)  d\mu(s) \\
	&= \sum_{j=0}^{M-1} \int_0^\tau q^{(\frac{1}{2}+is)(2j - |y|)} \int_{E_j(y)} \mathcal{A}_M(y,\omega,s)  d\nu(\omega)  d\mu(s) \\
	&\qquad + \int_0^\tau  A_M(y,s) \phi_{s}(y)  d\mu(s) \\
	&= \sum_{j=0}^{M} I_{j,M},
\end{align*}
where for all $j \in \{0, \ldots, M-1\}$
\[ I_{j,M} = \int_0^\tau q^{(\frac{1}{2}+is)(2j - |y|)} \int_{E_j(y)} \mathcal{A}_M(y,\omega,s)  d\nu(\omega)  d\mu(s), \]
and
\[ I_{M,M} = \int_0^\tau  A_M(y,s) \phi_{s}(y)  d\mu(s). \]
We first look at the last term $I_{M,M}$ of the sum. The spherical function $\phi_s(y)$ is given on $\R \setminus \tau \Z$ by
\[ \phi_s(y) = c(s) q^{(is - \frac{1}{2})|y|} + c(-s) q^{(-is - \frac{1}{2})|y|}. \]
Recall from $\eqref{e:plancherel}$ that $d\mu(s) = c_P |c(s)|^{-2} ds$. We integrate by parts $N$ times $I_{M,M}$. The boundary terms vanish because of condition \ref{bord} of the definition of the class of symbols. We have

\begin{align*}
I_{M,M} &= c_P \int_0^\tau  A_M(y,s) q^{(is - \frac{1}{2})|y|} c(s) |c(s)|^{-2} \\
& \quad \qquad + A_M(y,s) q^{(-is - \frac{1}{2})|y|} c(-s) |c(s)|^{-2}  ds \\
& = \frac{c_P q^{-\frac{|y|}{2}}}{(i|y|\log q)^N} \int_0^\tau q^{is|y|}  \partial_s^N \left( A_M(y,s)c(s)|c(s)|^{-2}\right) \\ 
& \quad \qquad - q^{-is|y|} \partial_s^N \left( A_M(y,s)c(-s)|c(s)|^{-2}\right)ds
\end{align*}
The derivative
\[ \partial_s^N \left( A_M(y,s)c(\pm s)|c(s)|^{-2}\right) \]
is a linear combination of $N+1$ terms of the form
\[ \partial_s^k A_M(y,s) \partial_s^{N-k} \left(c(\pm s)|c(s)|^{-2}\right). \]
Because  $\| \partial_s^k A_M(y,\cdot) \|_\infty \leq \| \partial_s^k a \|_\infty$ for all $k$ we have 
\begin{align*}
 & \left| \int_0^\tau q^{\mp is|y|}  \partial_s^k A_M(y,s) \partial_s^{N-k} \left(c(\pm s)|c(s)|^{-2}\right) ds \right| \\
 & \qquad \leq  \| \partial_s^k a \|_\infty \int_0^\tau \partial_s^{N-k} \left( |c(\pm s)||c(s)|^{-2} \right) ds.
\end{align*}
Using the fact that the poles of $c(\pm s)$ are compensated by $|c(s)|^{-2}$, as can be seen in $\eqref{c(z)}$ and $\eqref{c(s)^-2}$ and therefore that the map $s \mapsto c(\pm s)|c(s)|^{-2}$ is $C^\infty([0,\tau])$ we conclude that
\[ |I_{M,M}| \leq C_N \frac{q^{-\frac{|y|}{2}}}{(1+|y|)^N} \sum_{k=0}^{N} \| \partial_s^k a \|_\infty. \]

It remains to bound the terms $I_{j,M}$ for $0 \leq i \leq M-1$. After $N+1$ integration by parts each term is given by
\[ \frac{c_P q^{j - \frac{|y|}{2}}}{[i(2j - |y|) \log q]^{N+1}} \int_0^\tau q^{is(2j - |y|)} \partial_s^{N+1} \left( \int_{E_j(y)} \mathcal{A}_M(y,\omega,s) d\nu(\omega) |c(s)|^{-2} \right) ds,\]
which is a linear combination of $N+2$ terms of the form
\[ \frac{c_P q^{j - \frac{|y|}{2}}}{[i(2j - |y|) \log q]^{N+1}} \int_0^\tau q^{is(2j - |y|)}  \int_{E_j(y)} \partial_s^k \mathcal{A}_M(y,\omega,s) d\nu(\omega) \partial_s^{N+1-k} (|c(s)|^{-2}) ds,\]
the module of which is bounded from above by
\[ \frac{c_P q^{j - \frac{|y|}{2}}}{(2j - |y|)^{N+1} \log q^{N+1}}  \nu(E_j(y)) \|\partial_s^k a \|_\infty \int_0^\tau \partial_s^{N+1-k} (|c(s)|^{-2}) ds \]
\[\leq C_N \frac{q^{- \frac{|y|}{2}}}{(2j - |y|)^{N+1}} \|\partial_s^k a \|_\infty,\]
using the fact that $\nu(E_j(y)) \leq q^{-j}$ and $s \mapsto |c(s)|^{-2}$ is $C^\infty([0,\tau])$.

Now, because $0 \leq j \leq M$  and $\frac{|y|}{3} - 1 \leq M \leq \frac{|y|}{3}$ we have $|2j - |y|| \geq \frac{|y|}{3}$ and by adapting the constant we obtain
\[C_N \frac{q^{- \frac{|y|}{2}}}{(2j - |y|)^{N+1}} \|\partial_s^k a \|_\infty \leq C'_N \frac{q^{- \frac{|y|}{2}}}{(1 + |y|)^{N+1}} \|\partial_s^k a \|_\infty \]
Finally there is a constant $C_N$ such that
\[ |I_{j,M}| \leq C_N \frac{q^{- \frac{|y|}{2}}}{(1 + |y|)^{N+1}} \sum_{k=0}^{N+1} \|\partial_s^k a \|_\infty.  \]

We thus have
\[ \sum_{j=0}^{M-1} |I_{j,M}| \leq M  C_N \frac{q^{- \frac{|y|}{2}}}{(1 + |y|)^{N+1}} \sum_{k=0}^{N+1} \|\partial_s^k a \|_\infty. \]
But $M \leq C(1+|y|)$, so changing the constant $C_N$ we have
\[ \sum_{j=0}^{M-1} |I_{j,M}| \leq C_N \frac{q^{- \frac{|y|}{2}}}{(1 + |y|)^{N}} \sum_{k=0}^{N+1} \|\partial_s^k a \|_\infty. \]
\end{proof}

\section{Continuity of the operators}\label{s:continuity}
\begin{thm}\label{continuity_symbol}
The operator $\Opd(a)$ associated with a double symbol $a \in \S(\T \times \T)$ can be extended to a bounded operator from $L^2(\T)$ to $L^2(\T)$. The following inequality holds: there exists $C >0$ such that
\[ \|\Opd(a) \|_2 \leq C \left( \|a\|_{\Omega,4} + \sum_{k=0}^4 \|\partial_s^k a \|_\infty \right), \]
where
\[ \|a\|_{\Omega,4} = \sup_{x,y \in \T} \sup_{n \in \N} (n+1)^4 \| (a - \mathcal{E}^x_n a)(x,y, \cdot, \cdot) \|_\infty. \]
\end{thm}

This theorem is an immediate consequence of the following proposition and of proposition \ref{rapid decay proposition} on the rapid decay of the kernel associated with $\Opd(a)$.
\begin{prop}\label{continuity_kernel}
 Let $A : \mathcal{F}_c(\T) \rightarrow L^2(\T)$ be an operator, where $\mathcal{F}_c$ is the set of finitely supported functions on $\T$. Suppose that there exists an integer $N \geq 3$ such that the kernel of $A$ satisfies
 \begin{equation}\label{ineq_ker}
  |K_A(x,y)| \leq C_A(N) \frac{q^{-\frac{d(x,y)}{2}}}{(1+d(x,y))^N}.
 \end{equation}
 Then $A$ can be extended to a bounded operator from $L^2$ to $L^2$ and there exists a constant $C(N)$ such that
 \[ \|A \|_{L^2 \rightarrow L^2} \leq C_A(N) C(N), \]
 the constant $C_A(N)$ being the same as in inequality $\eqref{ineq_ker}$.
\end{prop}
\begin{proof}
 The following lemma will be used
\begin{lem}\label{kernel product lemma}
 Let $N \geq 3$. If there are constants $C_A = C_A(N)$ and $C_B = C_B(N)$ such that for all $x,y \in \T$ \[ |K_A(x,y)| \leq C_A \frac{q^{-d(x,y)/2}}{(1 + d(x,y))^N} \quad \text{ and } \quad |K_B(x,y)| \leq C_B \frac{q^{-d(x,y)/2}}{(1 + d(x,y))^N},\] then there is a constant $C(N)$ such that
\[ |K_{A\circ B}(x,y)| \leq C(N) C_A(N) C_B(N) \frac{q^{-d(x,y)/2}}{(1 + d(x,y))^N} \]
for all $x,y \in \T$.
\end{lem}

\begin{proof}
The kernel of the product $K_{A \circ B} (x,y) = \sum_z K_{A} (x,z) K_{B} (z,y)$ can be bounded by
\[ |K_{A \circ B} (x,y) | \leq C_A C_B \sum_z \frac{q^{-d(x,z)/2}}{(1 + d(x,z))^N} \frac{q^{-d(z,y)/2}}{(1 + d(z,y))^N} \]

On the tree $\T$, we have $d(x,z) + d(z,y) = d(x,y) + 2d(z,c_z(x,y))$, where $c_z(x,y)$ is the first vertex lying on both oriented segments $[x,z]$ and $[y,z]$. Note that this vertex lies on $[x,y]$. The sets $F_n = \{z \; | \; d(z,c_z(x,y)) = n \}$, $n \in \mathbb{N}$ then form a partition of the tree. In each $F_n$, for every $k = d(x,c_z(x,y))$ between $0$ and $d(x,y)$, the number of vertices $z$ such that $d(x,z) = k + n$ (i.e. such that $d(z,y) = d(x,y) + n - k$) is equal to $q^n$ when $k=0$ or $k=d(x,y)$, and is equal to $(q-1)q^{n-1}$ in the other cases. We then have
\begin{align*}
|K_{A \circ B} (x,y) | 
	& \leq C_A C_B q^{-d(x,y)/2} \sum_z \frac{q^{-c_z(x,y)}}{(1 + d(x,z))^N(1 + d(z,y))^N} \\
	& = C_A C_B q^{-d(x,y)/2} \sum_{n\geq 0} \sum_{z \in F_n} \frac{q^{-n}}{(1 + d(x,z))^N(1 + d(z,y))^N} \\
	& \leq C_A C_B q^{-d(x,y)/2} \sum_{n\geq 0} \sum_{k=0}^{d(x,y)} \frac{1}{(1 + k + n)^N(1 + d(x,y) + n - k)^N}
\end{align*}

To get the result, it is sufficient to prove that there exists a constant $C(N)$ such that for all $D \in \mathbb{N}$
\begin{equation}\label{serie lemme}
 \sum_{n \geq 0} \sum_{k=0}^{D} \frac{1}{(1 + k + n)^N(1 + D + n - k)^N} \leq \frac{C(N)}{(1+D)^N}.
\end{equation}
Let us denote by $\lfloor D/2 \rfloor$ the integral part of $D/2$. We have 
\[\sum_{n \geq 0} \sum_{k=0}^{D} \frac{1}{(1 + k + n)^N(1 + D + n - k)^N} \leq 2 \sum_{n \geq 0} \sum_{k=0}^{\lfloor D/2 \rfloor} \frac{1}{(1 + k + n)^N(1 + D + n - k)^N} \] where the inequality is an equality when $D$ is odd.
Now for every $n \geq 0$ and every $k \in \{0,\ldots,\lfloor D/2 \rfloor\}$, $D-k \geq D/2$ and
\[ \frac{1}{(1+D+n-k)} \leq \frac{1}{1+D/2}. \]
Therefore we have
\begin{align*}
 \sum_{n \geq 0} \sum_{k=0}^{D} \frac{1}{(1 + k + n)^N(1 + D + n - k)^N} 
 &\leq \frac{2}{(1+D/2)^N}\sum_{n \geq 0} \sum_{k=0}^{\lfloor D/2 \rfloor} \frac{1}{(1 + k + n)^N} \\
 &\leq \frac{2}{(1+D/2)^N}\sum_{n \geq 0} \sum_{k\geq0} \frac{1}{(1 + k + n)^N} 
\end{align*}
The series converges whenever $N \geq 3$, and $(1+D/2)^{-1} \leq 2(1+D)^{-1}$, so we obtain $\eqref{serie lemme}$, which concludes the proof.
\end{proof}

Notice that $A^*$, the adjoint of $A$, satisfies $\eqref{ineq_ker}$, because $K_{A^*}(x,y) = \overline{K_A(y,x)}$. Therefore $(A^*A)^k$ satisfies $\eqref{ineq_ker}$ with constant $C_A(N)^{2k} C(N)^{2k}$, by induction. 

We will first assume that $K_A(x,y)$ has finite support and then go back to the general case. 

Because $K_A(x,y)$ has finite support, both $K_A$ and $K_{A^*}$ have finite support in the first argument. The support of $K_{(A^*A)^k}$ in the first argument is thus finite\footnote{Note that a finite support in the first argument for $K_{A^*}$ is enough here, because then $K_{(A^*A)^k}$ has also finite support in the first argument}~:
 \[ \# \text{Supp}_1 = \# \{x \in \T \: | \: \exists y \in \T, K_{(A^*A)^k}(x,y) \neq 0 \} < \infty. \]
In this case $(A^*A)^k$ is Hilbert-Schmidt~: we have
\begin{align*}
 \|(A^*A)^k\|_2 \leq \|(A^*A)^k\|_{HS} &= \sum_{x,y \in \T} |K_{(A^*A)^k}(x,y)|^2 \\
  &\leq C_A(N)^{2k} C(N)^{2k} \sum_{x \in \T} \sum_{n \in \mathbb{N}} \sum_{y \in S(x,n)} \frac{q^{-n}}{(1+n)^{2N}} \\
  &\leq C_A(N)^{2k} C(N)^{2k}  \# \text{Supp}_1 \sum_{n \in \mathbb{N}} \frac{2}{(1+n)^{2N}} \\
  &\leq C_A(N)^{2k} C(N)^{2k} C'(N)  \# \text{Supp}_1 \\
\end{align*}
where on the third line we used the fact that there are $(q+1)q^{n-1} \leq 2q^n$ vertices on the sphere $S(x,n)$ of center $x$ and radius $n$, for each $x \in \T$. We then use the facts that $\|A^*A\|_2 = \|A\|^2_2$  and $A^*A$ is self-adjoint in order to write
\begin{align*}
 \|A\|_2 & = \|A^*A\|_2^{1/2} = \rho(A^*A)^{1/2} = \lim_{k \rightarrow \infty} \|(A^*A)^k \|_2^{\frac{1}{2k}} \\
  &\leq \lim_{k \rightarrow \infty} (C'(N) \# \text{Supp}_x)^{\frac{1}{2k}} C_A(N) C(N) \\
  &= C_A(N) C(N)
\end{align*}
where $\rho(A^*A)$ is the spectral radius of $A^*A$. We thus have a bound on $\|A\|_2$ which is independent of the support of $A$.

Going back to the general case, for an operator $A: \mathcal{F}_c(\T) \rightarrow L^2(\T)$ satisfying $\eqref{ineq_ker}$ we write
\[ A_R = M_{\chi_{B(0,R)}} A  M_{\chi_{B(0,R)}} \]
where $M_f$ is the multiplication by the function $f$ and $\chi_{B(0,R)}$ is the characteristic function of the ball of center $0$ and radius $R$. The operator $A_R$ has a kernel with finite support and converges weakly to $A$. 
Now $A_R$ satisfies $\eqref{ineq_ker}$ with constant $C_A(N) C(N)$ and therefore we have
\[ \langle \psi, A_R \varphi \rangle \leq C_A(N) C(N) \| \psi \|_2 \| \varphi \|_2, \]
the right member ot the inequality being independent of $R$, which gives when $R \rightarrow \infty$
\[ \langle \psi, A \varphi \rangle \leq C_A(N) C(N) \| \psi \|_2 \| \varphi \|_2 \]
for all $\varphi, \psi$ with finite support on $\T$. 
By a density argument we can replace $\psi$ with $A \varphi$ in the previous inequality and we get, for all $\varphi$ with finite support
\[ \| A \varphi \|_2 \leq C_A(N) C(N) \| \varphi \|_2. \]
The operator $A$ can thus be extended to a bounded operator from $L^2(\T)$ to $L^2(\T)$ with the same bound.
\end{proof}

\section{Adjoint and product}\label{s:adjoint}

\begin{thm}\label{t:adjoint}
Let $a = a_\epsilon \in \Sc$. Let $\Op(a)^*$ be the adjoint of $\Op(a)$. Then $\Op(a) - \Op^*(a)$ is negligible. In particular we have
\[ \| \Op(a)^* - \Op(\overline{a}) \|_{L^2 \rightarrow L^2} = o(1), \]
where $o(1) \rightarrow 0$ when $\epsilon \rightarrow 0$.
\end{thm}

\begin{proof}
We will work on the kernels in order to estimate the difference $\Op(a)^* - \Op(\overline{a})$.

If $K_a(x,y)$ is the kernel of $\Op(a)$, then $\overline{K_a(y,x)}$ is the kernel of the adjoint $\Op(a)^*$. We study the kernel of the remainder $K(x,y) = \overline{K_a(y,x)} - K_a(x,y)$.

As the rapid decay property away from the diagonal (proposition \ref{rapid decay proposition}) is symmetric in $x$ and $y$, the kernel of the adjoint also has it and we have
\[ |K(x,y)| \leq C_{N} \frac{q^{-\frac{d(x,y)}{2}}}{(1+d(x,y))^N} \]
for all $N \in \N$. Therefore if we cut off the kernel, for all $\rho > 0$
\[ K(x,y) = K(x,y)\chi_{\{d(x,y) \leq \rho\}} + K_{R}(x,y) \]
then $K_{R}(x,y) = K(x,y)\chi_{ \{d(x,y) > \rho\} }$ satisfies
\[ |K_{R}(x,y)| \leq C_{N} \frac{1}{(1+\rho)^\alpha}\frac{q^{-\frac{d(x,y)}{2}}}{(1+d(x,y))^{N-\alpha}}. \]
We take $N - \alpha \geq 3$ so that the operator $R$ associated with this kernel is bounded on $L^2(\T)$, according to proposition $\ref{continuity_kernel}$, and we have
\[ \|R\|_2 \leq C_{\alpha} \frac{1}{(1+\rho)^\alpha} \leq C_{\alpha} \rho^{-\alpha}. \]

The other part $K(x,y)\chi_{\{d(x,y) \leq \rho\}}$ is the kernel associated with the pseudo-differential operator of double symbol 
\[ r(x,y,\omega,s) = (\overline{a(y,\omega,s)} - \overline{a(x,\omega,s)})\chi_{\{d(x,y) \leq \rho\}}. \]
But $a \in \Sc$, and according to condition $\ref{varx}$ of definition \ref{d:symbol} of $\Sc$, $\forall l \in \N$ and $\forall x,y \in \T, \exists C_{l} > 0$ such that $\forall n \in \N$
\begin{align*}
 | (r - \mathcal{E}^x_n r)(x,y,\omega,s) | 
 &= | (a - \mathcal{E}^x_n a)(x,\omega,s) - (a - \mathcal{E}^x_n a)(y,\omega,s) |\chi_{\{d(x,y) \leq \rho\}} \\
 &\leq \epsilon  \frac{C_l(d(x,y))}{(1+n)^l} \chi_{\{d(x,y) \leq \rho\}} \\
 &\leq \epsilon  \frac{C_l(\rho)}{(1+n)^l}.
\end{align*}
We used that $t \mapsto C_l(t)$ is an increasing function. Therefore, $r$ satisfies condition \ref{reg_omega} (and \ref{bord}) of definition \ref{d:double} uniformly in $y$ and according to proposition \ref{rapid decay proposition} we have
\[ |K_r(x,y)| \leq \epsilon C_l(\rho) \frac{q^{-\frac{d(x,y)}{2}}}{(1+d(x,y))^l}, \]
where $K_r$ is the kernel of the pseudo-differential operator $\Op(r)$ associated with the double symbol $r$.
We take $l = 3$ and apply proposition \ref{continuity_kernel} in order to obtain
\[ \| \Op(r) \|_2 \leq C \epsilon C_3(\rho), \]
for some constant $C > 0$.

We thus have for every $\alpha,\rho > 0$ 
\[ \| \Op(a)^* - \Op(\overline{a}) \|_2 \leq C_{\alpha} (\epsilon C_3(\rho) + \rho^{-\alpha}).\]
We then take $\rho = \rho(\epsilon)$ tending to infinity when $\epsilon \rightarrow 0$ but sufficiently slowly such that $\epsilon C_3(\rho) \rightarrow 0$ when $\epsilon \rightarrow 0$.
\end{proof}

\begin{thm}\label{t:product}
 Let $a \in \S$ and $b = b_\epsilon \in \Sc$. Then $\Op(a)\Op(b) - \Op(ab)$ is negligible. In particular we have
 \[ \| \Op(a)\Op(b) - \Op(ab) \|_{L^2 \rightarrow L^2} = o(1) \]
 where $o(1) \rightarrow 0$ when $\epsilon \rightarrow 0$.
\end{thm}

\begin{proof}
 Recall that if we define $\tilde{a}(x,y,\omega,s) = a(x,\omega,s)$ for all $y \in \T$, then $\Opd(\tilde{a}) = \Op(a)$. We will call this the \emph{left quantization} and denote it by  $\Op_l(a)$. Now if $\hat{a}(x,y,\omega,s) = a(y,\omega,s)$ for all $x \in \T$, we can define another operator, the \emph{right quantization}, by writing $\Op_r(a) = \Opd(\hat{a})$. If $k_a(x,y)$ is the kernel of $\Op_l(a)$ :
 \[ k_a(x,y) = \int_\Omega \int_0^\tau q^{\left(\frac{1}{2}+is \right)(h_\omega(y) - h_\omega(x))} a(x,\omega,s) d\nu_x(\omega) d\mu(s) \]
 then the kernel of  $\Op_r(a)$ is
 \begin{align*}
  \int_\Omega \int_0^\tau & q^{\left(\frac{1}{2}+is \right)(h_\omega(y) - h_\omega(x))} a(y,\omega,s) d\nu_x(\omega) d\mu(s) \\
  & \quad = \int_\Omega \int_0^\tau q^{\left(\frac{1}{2}-is \right)(h_\omega(x) - h_\omega(y))} a(y,\omega,s) d\nu_y(\omega) d\mu(s) \\
  & \quad = \overline{k_{\bar{a}}(y,x)}
 \end{align*}
because $d\nu_x(\omega) = q^{(h_\omega(x) - h_\omega(y))}d\nu_y(\omega)$.
If $a$ is real, then $\Op_l(a) = \Op^*(a)$ and theorem \ref{t:adjoint} tells us that $\Op_l(a) - \Op_r(a)$ is negligible. If $a$ is complex, we can treat separately the real and imaginary parts and we get also that $\Op_l(a) - \Op_r(a)$ is negligible.

Let $c(x,y,\omega,s) = a(x,\omega,s)b(y,\omega,s)$. Then it can be checked that $\Op_l(a)\Op_r(b) = \Opd(c)$. Modulo negligible operators, we thus have
\begin{align*}
 \Op(a)\Op(b) &= \Op_l(a) \Op_l(b) \\
 &= \Op_l(a)\Op_r(b) \\
 &= \Opd(c) \\
 &= \Op(ab) + \Opd(r) 
\end{align*}
where $r(x,y,\omega,s) = a(x,\omega,s)(b(y,\omega,s) - b(x,\omega,s))$. As in the proof of theorem \ref{t:adjoint}, it is sufficient to study a cut-off of the remainder $\Opd(r)$. Indeed, lemma \ref{kernel product lemma} tells us that the kernel $K$ of $\Op(a)\Op(b)$ has the property of rapid decay
\[ \forall N \in \N, \exists C_N > 0, \quad |K(x,y)| \leq C_N \frac{q^{-\frac{d(x,y)}{2}}}{(1+d(x,y))^N} \]
 as a product of two operators satisfying this property, and the same is true for $\Op(ab)$, because $ab \in \Sc$ according to lemma \ref{product closure}.

If we denote by $K_r$ the kernel of $\Opd(r)$, for all $\rho > 0$ we have
\[ K_r(x,y) = K_r(x,y)\chi_{\{d(x,y) \leq \rho\} } + K_R(x,y) \]
where $K_R$ satisfies
\[ \forall N \in \N, \exists C_N > 0, \quad |K_R(x,y)| \leq C_N \frac{q^{-\frac{d(x,y)}{2}}}{(1+d(x,y))^N} \rho^{-\alpha} \]
and therefore corresponds to a negligible operator, provided that we choose $\rho = \rho(\epsilon)$ such that $\lim_{\epsilon \rightarrow 0} \rho^{-\alpha}(\epsilon) = 0$.

The part $K_r(x,y)\chi_{\{d(x,y) \leq \rho\} }$ is the kernel of the operator $\Opd(r_\rho)$, where
\[ r_\rho(x,y,\omega,s) = r(x,y,\omega,s)\chi_{\{d(x,y) \leq \rho\} } = a(x,\omega,s)(b(y,\omega,s) - b(x,\omega,s)) \chi_{\{d(x,y) \leq \rho\} } \]
Given that $a$ is bounded and $b \in \Sc$, we have, using condition $\ref{varx}$ of definition \ref{d:symbol} of $\Sc$, $\forall l \in \N$ and $\forall x,y \in \T, \exists C_{l} > 0$ such that $\forall n \in \N$
\begin{align*}
 | (r_\rho - \mathcal{E}^x_n r_\rho)(x,y,\omega,s) | 
 &\leq \|a\|_\infty | (b - \mathcal{E}^x_n b)(x,\omega,s) - (b - \mathcal{E}^x_n b)(y,\omega,s) |\chi_{\{d(x,y) \leq \rho\}} \\
 &\leq C_{a}  \epsilon  \frac{C_l(d(x,y))}{(1+n)^l} \chi_{\{d(x,y) \leq \rho\}} \\
 &\leq C_{a}  \epsilon  \frac{C_l(\rho)}{(1+n)^l}.
\end{align*}
We used that $t \mapsto C_l(t)$ is an increasing function. The rest of the proof is then similar to the proof of theorem \ref{t:adjoint} on the adjoint.The function $r_\rho$ satisfies the conditions of definition \ref{d:double} and is therefore in the class $\S$. According to proposition \ref{rapid decay proposition} we have for every $l \in \N$,
\[ |K_{r_\rho}(x,y)| \leq C_a \epsilon C_l(\rho) \frac{q^{-\frac{d(x,y)}{2}}}{(1+d(x,y))^l}, \]
where $K_{r_\rho}$ is the kernel of $\Opd(r_\rho)$.
We take $l = 3$ and apply proposition \ref{continuity_kernel} in order to obtain
\[ \| \Opd(r_\rho) \|_2 \leq C_a \epsilon C_3(\rho). \]

We thus have for every $\alpha,\rho > 0$ 
\[ \| \Op(a)\Op(b) - \Op(ab) \|_2 \leq C_{a,\alpha} (\epsilon C_3(\rho) + \rho^{-\alpha}).\]
We then take $\rho = \rho(\epsilon)$ tending to infinity when $\epsilon \rightarrow 0$ but sufficiently slowly such that $\epsilon C_3(\rho) \rightarrow 0$ when $\epsilon \rightarrow 0$.
\end{proof}

With the preceding proof, we obtain a remainder in the product formula of the order of $O_\delta(\epsilon^{1-\delta})$ for any $\delta > 0$. By following a different strategy, it is possible to have a better remainder. But we have to make a stronger hypothesis on one of the symbols. However, this hypothesis is satisfied by our principal example of symbols (see Example \ref{ex:main}).

\begin{thm}\label{t:product optimal}
Let $a \in \S$ and $b=b_\epsilon \in \Sc$. Recall condition \ref{varx cross derivative} of definition \ref{d:symbol} as satified by symbol $b$:
for every $l \in \N$, there exists a function $t \mapsto C_l(t)$ such that for every $x,y \in \T$ and $n \in \N$
\[ | (b - \mathcal{E}^x_n b)(x,\omega,s) - (b - \mathcal{E}^x_n b)(y,\omega,s) | \leq \epsilon  \frac{C_l(d(x,y))}{(1+n)^l}, \]
and suppose that for the symbol $b$ the function $t \mapsto C_l(t)$ is polynomial.
In this case we have
\[ \| \Op(a)\Op(b) - \Op(ab) \|_2 \leq C \epsilon. \]
\end{thm}

\begin{proof}
The action of the composition $\Op(a)\Op(b)$ of two pseudo-differential operators, is equivalent to the action of an operator $\Op(a\#b)$ where $a\#b(x, \omega, s)$ is obtained from the kernel by
\begin{align*}
 a\#b(x,\omega,s) &= \sum_{z \in \T} q^{(\frac{1}{2}-is)(h_\omega(z) - h_\omega(x))} K_{a\#b}(x,z) \\
 &= \sum_{z \in \T} q^{(\frac{1}{2}-is)(h_\omega(z) - h_\omega(x))} \sum_{y \in \T} K_a(x,y)K_b(y,z) \\
 &= \sum_{y \in \T} K_a(x,y) q^{(\frac{1}{2}-is)(h_\omega(y) - h_\omega(x))} \sum_{z \in \T} q^{(\frac{1}{2}-is)(h_\omega(z) - h_\omega(y))} K_b(y,z) \\
 &= \sum_{y \in \T} K_a(x,y) q^{(\frac{1}{2}-is)(h_\omega(y) - h_\omega(x))} b(y,\omega,s).
\end{align*}
The operator $\Op(a)\Op(b) - \Op(ab) = \Op(a\#b) - \Op(ab)$ is then an operator $\Op(r)$, where
\begin{align}
r(x,\omega,s) &= a\#b(x,\omega,s) - (ab)(x,\omega,s) \nonumber\\
	&= \sum_{y \in \T} K_a(x,y) q^{(1/2 - is)( h_\omega(y) - h_\omega(x))} (b(y,\omega,s) - b(x,\omega,s)),\label{remainder symbol product optimal}
\end{align}
because
\[ \sum_{y \in \T} K_a(x,y) q^{(1/2 - is)( h_\omega(y) - h_\omega(x))} b(x,\omega,s) = a(x,\omega,s)b(x,\omega,s). \]

We would like to show that $r \in \S$ (when seen as a double symbol) and that the $L^2$-norm of $\Op(r)$ is bounded by $\epsilon$. In order to do that, according to theorem \ref{continuity_symbol}, we must first bound the derivatives of $r(x,\omega,s)$ with respect to $s$.

Note that for every $k \in \N$,
\begin{align*}
 |\partial_s^k  q^{(\frac{1}{2}-is)( h_\omega(y) - h_\omega(x))}|
  &= |(h_\omega(y) - h_\omega(x))^k q^{(\frac{1}{2}-is)( h_\omega(y) - h_\omega(x))}| \\
  &\leq d(x,y)^k q^{\frac{1}{2}( h_\omega(y) - h_\omega(x))}.
\end{align*}
We thus have for every $\alpha \in \N$, using the Leibniz rule on \eqref{remainder symbol product optimal}, and the previous inequality
\begin{align}
|\partial_s^\alpha r(x,\omega,s)|
	& \leq C_\alpha \sum_{k=0}^\alpha \sum_{y \in \T} |K_a(x,y)| d(x,y)^k q^{\frac{1}{2}( h_\omega(y) - h_\omega(x))} |\partial_s^{\alpha - k} (b(y,\omega,s) - b(x,\omega,s))| \nonumber\\
	& \leq C_\alpha C_a(N) \sum_{k=0}^\alpha \sum_{y \in \T} \frac{q^{-\frac{d(x,y)}{2}}}{(1+d(x,y))^N} d(x,y)^{k} q^{\frac{1}{2}( h_\omega(y) - h_\omega(x))} \epsilon d(x,y) \nonumber \\
	& \leq C_\alpha C_a(N) \sum_{k=0}^\alpha \sum_{y \in \T} \frac{q^{-\frac{d(x,y)}{2}}}{(1+d(x,y))^{N-k-1}} q^{\frac{1}{2}( h_\omega(y) - h_\omega(x))} \epsilon \label{eq product optimal 2},
\end{align}
for all $N \in \N$. We used proposition \ref{rapid decay proposition} for symbol $a$ on the rapid decay of the kernel of an operator with symbol in $\S$, and condition \ref{varx} in the definition of $\Sc$ for symbol $b$, controlling the variation with respect to $x$ of a semi-classical symbol. 

To complete this estimation we will use the equality
\[ \sum_{y:d(x,y)=n} q^{\frac{1}{2} ( h_\omega(y) - h_\omega(x))} = (1+n) q^{\frac{n}{2}}. \]
We have for every $\alpha \in \N$ and $N \in \N$
\begin{align}
|\partial_s^\alpha r(x,\omega,s)|
	& \leq C_\alpha C_a(N)\epsilon \sum_{k=0}^\alpha \sum_{n \in \N} \frac{q^{-\frac{n}{2}}}{(1+n)^{N-k-1}}  \sum_{y:d(x,y)=n} q^{\frac{1}{2} ( h_\omega(y) - h_\omega(x))} \label{e:optimal s derivative}\\
	& \leq C_\alpha C_a(N)\epsilon \sum_{k=0}^\alpha \sum_{n \in \N} \frac{q^{-\frac{n}{2}}}{(1+n)^{N-k-1}} (1+n) q^{\frac{n}{2}} \nonumber\\
	& \leq C_\alpha C_a(N)\epsilon \sum_{k=0}^\alpha \sum_{n \in \N} \frac{1}{(1+n)^{N-k-2}} \nonumber\\
	& \leq C_{\alpha,a} \epsilon,\nonumber
\end{align}
if we take $N \geq \alpha + 4$.

We must also bound, for every $l$
\[ \sup_n (1+n)^l |r - \mathcal{E}_n^x r|(x,\omega,s). \]
We have
\begin{align*}
(r - \mathcal{E}_n^xr)(x,\omega,s) &= \sum_{y \in \T} K_a(x,y)  \Bigg( q^{(1/2 - is)( h_\omega(y) - h_\omega(x))} (b(y,\omega,s) - b(x,\omega,s)) \\
&\quad - \mathcal{E}^x_n \left(q^{(\frac{1}{2} - is)( h_\omega(y) - h_\omega(x))}(b(y,\omega,s) - b(x,\omega,s))\right) \Bigg). 
\end{align*}
Recall that $c^x(y,\omega)$ is the confluence point of $y$ and $\omega$ centered in $x$, that is the last point lying on $[x,\omega)$ of the geodesic segment $[x, y]$. According to \eqref{e:confluence}, we have  $h_\omega(y) - h_\omega(x) = 2 d(x,c^x(y,\omega)) - d(x,y)$. We will divide into two parts, whether $d(x,c^x(y,\omega)) = n$ or not. 

\paragraph*{}
If $d(x,c^x(y,\omega)) \neq n$ we have
\[ \mathcal{E}^x_n \left(q^{(\frac{1}{2} - is)( h_\omega(y) - h_\omega(x))}(b(y,\omega,s) - b(x,\omega,s))\right) =  q^{(\frac{1}{2} - is)( h_\omega(y) - h_\omega(x))} \mathcal{E}^x_n (b(y,\omega,s) - b(x,\omega,s)). \]
Indeed, using the definition \eqref{e:mean operator} of $\mathcal{E}^x_n$ and the decomposition along the family of sets $\{ E^x_j(y) \}_j $ defined by \eqref{e:decomposition omega} we have
\begin{align*} 
\mathcal{E}^x_n &(q^{(\frac{1}{2} - is)( h_\omega(y) - h_\omega(x))}(b(y,\omega,s) - b(x,\omega,s))) \\
 &= \frac{1}{\nu_x(\Omega_n(x,\omega))} \sum_{j=0}^{d(x,y)} q^{(\frac{1}{2} - is)(2j - d(x,y))} \int_{\Omega_n(x,\omega) \cap E^x_j(y)} b(y,\omega',s) - b(x,\omega',s) d\nu_x(\omega') \\
 &= q^{(\frac{1}{2} - is)( h_\omega(y) - h_\omega(x))} \frac{1}{\mu(\Omega_n(x,\omega))}  \int_{\Omega_n(x,\omega)} b(y,\omega',s) - b(x,\omega',s) d\nu_x(\omega')
\end{align*}
because if $d(x,c^x(y,\omega)) \neq n$, $E^x_j(y) \cap \Omega_n(x,\omega) = \emptyset$, unless $ j = d(x,c^x(y,\omega)) $, in which case $2j - d(x,y) = h_\omega(y) - h_\omega(x)$ and $E^x_j(y) \cap \Omega_n(x,\omega) = \Omega_n(x,\omega)$.

This leads for this part of $(r - \mathcal{E}_n^x r)(x,\omega,s)$ to an expression of the form
\begin{equation*}
 \sum_{y:d(x,c^x(y,\omega)) \neq n} K_a(x,y) q^{(1/2 - is)( h_\omega(y) - h_\omega(x))} f(x,y,\omega,s),
\end{equation*}
where
\[ f(x,y,\omega,s) =  b(y,\omega,s) - b(x,\omega,s) - \mathcal{E}^x_n(b(y,\omega,s) - b(x,\omega,s)). \]
From the hypothesis of the theorem, we have
\[ |f(x,y,\omega,s)| \leq C_l(d(x,y))  \epsilon  \frac{1}{(1+n)^l} \]
with $t \mapsto C_l(t)$ polynomial in $t$.
We thus have
\begin{align*}
  (1+n)^l \Bigg| \sum_{y:d(x,c^x(y,\omega)) \neq n} & K_a(x,y) q^{(1/2 - is)( h_\omega(y) - h_\omega(x))} f(x,y,\omega,s) \Bigg| \\
  &\leq \sum_{y\in \T} |K_a(x,y)| q^{\frac{1}{2}( h_\omega(y) - h_\omega(x))} C_l(d(x,y)) \epsilon \\
  &\leq C_a(N) \epsilon \sum_{y\in \T} \frac{q^{-\frac{d(x,y)}{2}}}{(1+d(x,y))^N} q^{\frac{1}{2}( h_\omega(y) - h_\omega(x))} C_l(d(x,y)), 
\end{align*}
which is very similar to \eqref{eq product optimal 2}. We can therefore follow the same steps, and use the fact that the polynomial dependence of $C_l(d(x,y))$ in $d(x,y)$ is compensated by the rapid decay of $K_a(x,y)$, if we take $N$ sufficiently large. We obtain finally
\[ (1+n)^l \Bigg| \sum_{y:d(x,c^x(y,\omega)) \neq n} K_a(x,y) q^{(\frac{1}{2} - is)( h_\omega(y) - h_\omega(x))} f(x,y,\omega,s) \Bigg| \leq C_a \epsilon \]

\paragraph*{}
If $d(x,c^x(y,\omega)) = n$ then $d(x,y) \geq n$ and we have
\[ |K_a(x,y)| \leq C_a(N) \frac{1}{(1+n)^l} \frac{1}{(1+d(x,y))^{N-l}}. \]
Let $g(x,y,\omega,s) = q^{(\frac{1}{2} - is)( h_\omega(y) - h_\omega(x))}(b(y,\omega,s) - b(x,\omega,s))$. We have
\begin{equation}\label{e:optimal 10}
 (1+n)^l \left| \sum_{y : d(x,c^x(y,\omega)) = n} K_a(x,y)\left(g(x,y,\omega,s) - \mathcal{E}^x_n g(x,y,\omega,s)\right) \right|
\end{equation}
\begin{align*}
\leq C_a(N) \sum_{y : d(x,c^x(y,\omega)) = n}   \frac{q^{-\frac{d(x,y)}{2}}}{(1+d(x,y))^{N-l}} \Bigg( q^{\frac{1}{2}( h_\omega(y) - h_\omega(x))} |b(y,\omega,s) - b(x,\omega,s)| \\ +  \frac{1}{\nu_x(\Omega_n(x,\omega))}\int_{\Omega_n(x,\omega)} q^{\frac{1}{2}( h_{\omega'}(y) - h_{\omega'}(x))} |b(y,\omega',s) - b(x,\omega',s)| d\nu_x(\omega') \Bigg).
\end{align*}

The first part of the sum between brackets leads to a computation similar to \eqref{eq product optimal 2}: we know we have 
\[ C_a(N) \sum_{y : d(x,c^x(y,\omega)) = n}   \frac{q^{-\frac{d(x,y)}{2}}}{(1+d(x,y))^{N-l}} q^{\frac{1}{2}( h_\omega(y) - h_\omega(x))} |b(y,\omega,s) - b(x,\omega,s)| \]
\[ \leq C_a(N) \sum_{y \in \T}   \frac{q^{-\frac{d(x,y)}{2}}}{(1+d(x,y))^{N-l}} q^{\frac{1}{2}( h_\omega(y) - h_\omega(x))} |b(y,\omega,s) - b(x,\omega,s)| \]
\[ \leq C_a(l) \epsilon\]

The second part of the sum can be written
\begin{align*} 
 \frac{1}{\nu_x(\Omega_n(x,\omega))} \sum_{j=0}^{d(x,y)} q^{j - \frac{d(x,y)}{2}} \int_{\Omega_n(x,\omega) \cap E^x_j(y)} |b(y,\omega',s) - b(x,\omega',s)| d\nu_x(\omega') \\
=  \frac{1}{\nu_x(\Omega_n(x,\omega))} \sum_{j=n}^{d(x,y)} q^{j - \frac{d(x,y)}{2}} \int_{E^x_j(y)} |b(y,\omega',s) - b(x,\omega',s)| d\nu_x(\omega')
\end{align*}
noticing that when $d(x,c^x(y,\omega)) = n$, $E^x_j(y) \cap \Omega_n(x,\omega) = \emptyset $ unless $ n \leq j \leq d(x,y)$ in which case $E^x_j(y) \cap \Omega_n(x,\omega) = E^x_j(y)$.
We can then bound this by
\[\sum_{j=n}^{d(x,y)} \frac{q^{j - \frac{d(x,y)}{2}}}{\nu_x(\Omega_n(x,\omega))} \nu_x(E^x_j(y)) \epsilon d(x,y) \]
\[ \leq \sum_{j=n}^{d(x,y)} \frac{q^{- \frac{d(x,y)}{2}}}{\nu_x(\Omega_n(x,\omega))} \epsilon d(x,y) \]
where we used condition \ref{varx} of the definition \ref{d:symbol} of the symbol class, and the fact that $\nu_x(E^x_j(y)) \leq q^{-j}$.

Using the inequality $\nu_x(\Omega_n(x,\omega))^{-1} \leq C_q q^n$ and defining \[ B_k = \{ y : d(x,y)=k, d(x,c^x(y,\omega)) = n\}, \] we have, going back to \eqref{e:optimal 10}
\begin{align*}
\sum_{y : d(x,c^x(y,\omega)) = n}  & \frac{q^{-\frac{d(x,y)}{2}}}{(1+d(x,y))^{N-l}} \sum_{j=n}^{d(x,y)} \frac{q^{- \frac{d(x,y)}{2}}}{\nu_x(\Omega_n(x,\omega))} \epsilon d(x,y) \\
&\leq C_q \sum_{y : d(x,c^x(y,\omega)) = n} \frac{q^{-d(x,y)}}{(1+d(x,y))^{N-l-2}} q^n \epsilon \\
&=  C_q \sum_{k \geq n} \sum_{y \in B_k} \frac{q^{-k}}{(1+k)^{N-l-2}} q^n \epsilon.
\end{align*}
Finally, because $ \#B_k \leq C_q q^{k-n}$ we have 
\[ C_q \sum_{k \geq n} \sum_{y \in B_k} \frac{q^{-k}}{(1+k)^{N-l-2}} q^n \epsilon \leq  \left( C_q \sum_{k \geq n} \frac{1}{(1+k)^{N-l-2}}\right) \epsilon  \]
and
\[ (1+n)^l \left| \sum_{y : d(x,c^x(y,\omega)) = n} K_a(x,y)[ g(x,y,\omega,s) - \mathcal{E}^x_n g(x,y,\omega,s)] \right| \leq C(a,l,q) \epsilon. \]
We thus have for every $l \in \N$, and for all $x \in \T$
\[ \sup_n (1+n)^l \|(r - \mathcal{E}_n^x r)(x,\cdot,\cdot)\|_\infty \leq C(a,l,q) \epsilon. \]
Together with \eqref{e:optimal s derivative} and using theorem \ref{continuity_symbol}, we get the desired result.
\end{proof}

\section{Commutator with the Laplacian}\label{s:commutator}
In the usual pseudo-differential calculus on manifolds, the highest order term in the expansion of the symbol of the commutator $[\Op(a),\Op(b)]$ of two pseudo-differential operators is given by the Poisson bracket $\{a,b\}$. In our case it is not clear what would be the analogue of this quantity. We will limit ourselves to the special case of the commutator of a pseudo-differential operator with the Laplacian, defined for every function $f : \T \rightarrow \C$ by
\[ \Delta f(x) = \frac1{q+1}\sum_{y: d(x,y)=1} f(y). \]
This is the starting point of an ``Egorov''-type theorem, which gives invariance properties relating pseudo-differential operators and the dynamics on regular graphs.

\begin{prop}
Let $a \in \S$ be a symbol. We assume for simplicity that $a(x,y,\omega,s) = a(x,\omega,s)$ does not depend on $y$.
The commutator
\[ [\Delta, \Op(a)] = \Delta \Op(a) - \Op(a) \Delta \]
is an operator $\Op(c)$ where $c$ is given by 
\[ c = \frac{q^{\frac{1}{2}}}{q+1} \left( q^{-is}(a \circ \sigma - a) + q^{is} \left( La - a \right) \right). \]
We have more explicitly that $c(x,\omega,s)$ is equal to
\[ \frac{q^{\frac{1}{2}}}{q+1} \left( q^{-is}(a \circ \sigma(x,\omega,s) - a(x,\omega,s)) + \frac{q^{is}}{q} \sum_{y:\sigma_\omega(y) = x} \left( a(y,\omega,s) - a(x,\omega,s) \right) \right). \]
According to propositions \ref{prop stab geodesic} and \ref{prop stab transfer}, $c \in \S$ and the commutator is still a pseudo-differential operator.

The same result is true if we take $a \in \Sc$. In this case $c \in \Sc$.
\end{prop}

\begin{proof}
Let us first compute the symbol of $\Delta \Op(a)$. Recall that
\begin{align*}
 \Op(a)u(x) &= \sum_{y \in \T} \int_\Omega \int_0^\tau q^{(\frac{1}{2} + is)( h_\omega(y) - h_\omega(x))} a(x,\omega,s) u(y) d\nu_x(\omega)d\mu(s) \\
 &= \sum_{y \in \T} \int_\Omega \int_0^\tau q^{(\frac{1}{2} - is)( h_\omega(x) - h_\omega(y))} a(x,\omega,s) u(y) d\nu_y(\omega)d\mu(s).
\end{align*}
We have
\[ \Delta \Op(a)(u)(x) = \sum_{y \in \T} \int_\Omega \int_0^\tau \Delta \left( q^{(\frac{1}{2} - is)( h_\omega(\cdot) - h_\omega(y))} a(\cdot,\omega,s)\right)(x) u(y) d\nu_y(\omega)d\mu(s).  \]
Fix $\omega$ and $s$, and let $g(x) = q^{(\frac{1}{2}-is)(h_\omega (x) - h_\omega(y))} a(x,\omega,s)$. We split $\Delta g(x) = \frac{1}{q+1}\sum_{d(x,y)=1} g(y)$ into two parts depending on the direction of the shift:
\begin{align*}
 \Delta g(x) 
 &= q^{(\frac{1}{2}-is)(h_\omega(\sigma_\omega(x)) - h_\omega(y))} a\circ\sigma(x,\omega,s) \\
	&\quad + \frac1q \sum_{z:\sigma_\omega(z)=x} q^{(\frac{1}{2}-is)(h_\omega(z) - h_\omega(y))} a(z,\omega,s).
\end{align*}
where $\sigma_\omega$ is as defined in \ref{shift operator}. Then we can use the fact that $h_\omega(\sigma_\omega(x)) = h_\omega(x)+1$ to write
\begin{align*}
q^{(\frac{1}{2}-is)(h_\omega(\sigma_\omega(x)) - h_\omega(y))} a\circ\sigma(x,\omega,s)
= q^{(\frac{1}{2}-is)} q^{(\frac{1}{2}-is)(h_\omega(x) - h_\omega(y))} a\circ\sigma(x,\omega,s).
\end{align*}
We use the fact that $h_\omega(z) = h_\omega(x)-1$ when $\sigma_\omega(z) = x$ in the second part of the sum to obtain
\[\sum_{z:\sigma_\omega(z)=x} q^{(\frac{1}{2}-is)(h_\omega(z) - h_\omega(y))} a(z,\omega,s)
= q^{(\frac{1}{2}+is)} \frac1q \sum_{z:\sigma_\omega(z)=x} q^{(\frac{1}{2}-is)(h_\omega(x) - h_\omega(y))} a(z,\omega,s)\]
The symbol of $\Delta\Op(a)$ is thus given by
\[ \frac{1}{q+1}\left(q^{(\frac{1}{2}-is)}a\circ\sigma(x,\omega,s) + q^{(\frac{1}{2}+is)} \frac1q \sum_{z:\sigma_\omega(z)=x} a(z,\omega,s)\right) \]
\paragraph*{}
Let us now compute the symbol of $\Op(a)\Delta$. The kernel of $\Op(a)\Delta$ is given by
\[ K(x,y) = \frac{1}{q+1} \sum_{z:d(z,y)=1} k_a(x,z), \]
where $k_a$ is the kernel of $\Op(a)$.
\begin{align*}
 K(x,y) &= \frac{1}{q+1} \sum_{z:d(z,y)=1} \int_\Omega \int_0^\tau q^{(\frac{1}{2} + is)( h_\omega(z) - h_\omega(x))} a(x,\omega,s) u(y) d\nu_x(\omega)d\mu(s) \\
 &= \frac{1}{q+1} \sum_{z:d(z,y)=1} q^{(\frac{1}{2} + is)( h_\omega(z) - h_\omega(y))} \int_\Omega \int_0^\tau q^{(\frac{1}{2} + is)( h_\omega(y) - h_\omega(x))} a(x,\omega,s) u(y) d\nu_x(\omega)d\mu(s) \\
 &= \frac{1}{q+1} \left( q^{(\frac{1}{2} + is)} + q^{(\frac{1}{2} - is)} \right) \int_\Omega \int_0^\tau q^{(\frac{1}{2} + is)( h_\omega(y) - h_\omega(x))} a(x,\omega,s) u(y) d\nu_x(\omega)d\mu(s).
\end{align*}
Thus by the inversion formula the symbol of $\Op(a)\Delta$ is equal to
\[ \frac{1}{q+1}\left( q^{(\frac{1}{2}+is)}+q^{(\frac{1}{2}-is)} \right)a(x,\omega,s), \] 
and by substracting this from the symbol of $\Delta\Op(a)$ we obtain the symbol of the commutator.
\end{proof}

\bibliographystyle{alpha}
\bibliography{pseudo}

\end{document}